\documentclass[11pt]{article}

\usepackage[applemac]{inputenc} 

\usepackage{amsthm,amssymb,amsbsy,amsmath,amsfonts,amssymb,amscd,mathrsfs}
\usepackage[normalem]{ulem}

\long\def\symbolfootnote[#1]#2{\begingroup%
\def\thefootnote{\fnsymbol{footnote}}\footnote[#1]{#2}\endgroup} 

\usepackage{graphicx}
\DeclareGraphicsExtensions{.jpg}
\usepackage{color}
\usepackage{fullpage}
\newtheorem{theorem}{Theorem}
\newtheorem{corollary}[theorem]{Corollary}
\newtheorem{lemma}[theorem]{Lemma}
\newtheorem{proposition}[theorem]{Proposition}
\newtheorem{THM}[theorem]{Theorem}

\theoremstyle{definition} 
\newtheorem{definition}[theorem]{Definition}
\newtheorem{example}[theorem]{Example}
\theoremstyle{remark}
\newtheorem{remark}[theorem]{Remark}

\newcommand{\bt}{\begin{theorem}}
\newcommand{\et}{\end{theorem}}
\newcommand{\bl}{\begin{lemma}}
\newcommand{\el}{\end{lemma}}
\newcommand{\bp}{\begin{proposition}}
\newcommand{\ep}{\end{proposition}}
\newcommand{\bc}{\begin{corollary}}
\newcommand{\ec}{\end{corollary}}
\newcommand{\bdeff}{\begin{definition}}
\newcommand{\edeff}{\end{definition}}
\newcommand{\brem}{\begin{remark}}
\newcommand{\erem}{\end{remark}}
\newcommand{\bex}{\begin{example}}
\newcommand{\eex}{\end{example}}

\renewcommand{\r}[1]{(\ref{#1})}
\newcommand{\con}{{\mathcal C}}
\newcommand{\bR}{\mathbb{R}}
\newcommand{\be}{\begin{equation}}
\newcommand{\ee}{\end{equation}}

\newcommand{\da}{\searrow}
\newcommand{\ip}[2]{\left\langle#1,#2\right\rangle}
\newcommand{\lp}{\left(}
\newcommand{\rp}{\right)}
\newcommand{\lb}{\left[}
\newcommand{\rb}{\right]}
\newcommand{\lc}{\left\{}
\newcommand{\rc}{\right\}}

\newcommand{\Lap}{\Delta}

\newcommand{\E}{\mathbf{U}}

\newcommand{\GE}{N(\Gamma)}

\DeclareMathOperator{\Cut}{Cut}
\DeclareMathOperator{\Con}{Con}

\renewcommand{\phi}{\varphi}

\newcommand{\bi}{\begin{itemize}}
\newcommand{\iii}{\item}
\newcommand{\ei}{\end{itemize}}
\newcommand{\bd}{\begin{description}}
\newcommand{\ed}{\end{description}}

\newcommand{\bqn}{\begin{eqnarray}}
\newcommand{\eqn}{\end{eqnarray}}
\newcommand{\eqnn}{\nonumber\end{eqnarray}}

\newcommand{\nn}{\nonumber}
\newcommand{\ba}[1]{\begin{array}{#1}}
\newcommand{\ea}{\end{array}}
\newcommand{\wh}[1]{\widehat{#1}}

\newcommand{\lam}{\lambda}

\newcommand{\g}{\gamma}
\newcommand{\al}{\alpha}
\newcommand{\eps}{\varepsilon}

\newcommand{\R}{\mathbb{R}}
\newcommand{\N}{\mathbb{N}}
\newcommand{\C}{\mathbb{C}}

\newcommand{\mc}[1]{\mathcal{ #1 }}

\newcommand{\all}{\forall\,}

\newcommand{\la}{\langle}
\newcommand{\ra}{\rangle}

\newcommand{\VecM}{\mathrm{Vec}(M)}
\newcommand{\virg}[1]{``#1''}
\newcommand{\tx}[1]{\mathrm{#1}}
\newcommand{\til}[1]{\widetilde{#1}}

\newcommand{\distr}{\mc{D}}

\newcommand{\metr}{\textsl{g}}

\newcommand{\Pg}[1]{\left\{ #1 \right\}}

\newcommand{\hp}{hypothesis}
\newcommand{\EXP}{\mc{E}}
\newcommand{\Exp}{\mc{E}}
\newcommand{\lapl}{\Delta}
\newcommand{\dive}{\text{div}}
\newcommand{\grad}{\nabla}

\newcommand{\HH}{\mc{H}}
\newcommand{\tcon}{t_{con}}
\newcommand{\tcut}{t_{cut}}

\newcommand{\rvd}{rank-varying distribution}
\newcommand{\f}{f}
\newcommand{\rvsR}{rank-varying sub-Riemannian structure}

\input xy
 \xyoption{all}
 
\newcommand{\bD}{\Delta}
\newcommand{\Gq}{{{\bf G}}_q}

\begin{document}
\begin{center} \noindent
{\LARGE{\sl{\bf Small time heat kernel asymptotics at the sub-Riemannian cut locus}}}
\vskip 0.6 cm
Davide Barilari\\ 
{\footnotesize  CNRS, CMAP Ecole Polytechnique and equipe INRIA GECO Saclay-\^Ile-de-France, Paris, France  {\tt barilari@cmap.polytechnique.fr}}\\
\vskip 0.3cm
Ugo Boscain
\symbolfootnote[0]{This research has been supported by the European Research Council, ERC StG 2009 \virg{GeCoMethods}, contract number 239748, by the ANR Project GCM, program \virg{Blanche}, project number NT09-504490 and by the DIGITEO project CONGEO.}
\\
{\footnotesize CNRS, CMAP Ecole Polytechnique and equipe INRIA GECO  Saclay-\^Ile-de-France, Paris, France  {\tt boscain@cmap.polytechnique.fr}}
\vskip 0.3cm
Robert W.\ Neel\\
{\footnotesize Department of Mathematics, Lehigh University, Bethlehem, PA, USA \\ {\tt robert.neel@lehigh.edu}}\\
\vskip 0.5cm
\today
\end{center}

\vskip 0.3 cm
\begin{abstract} 
 For a sub-Riemannian manifold provided with a smooth volume, we relate the small time asymptotics of the heat kernel at a point $y$ of the cut locus from $x$ with roughly ``how much'' $y$ is conjugate to $x$. This is done under the hypothesis that all minimizers connecting $x$ to $y$ are strongly normal, i.e.\ all pieces of the trajectory are not abnormal. Our result is a refinement of the one of Leandre  $4t\log p_t(x,y)\to -d^2(x,y)$ for $t\to 0$, in which only the leading exponential term is detected. Our results are obtained by extending an idea of Molchanov
from the Riemannian to the sub-Riemannian case, and some details we get appear to be new even in the Riemannian context. 
These results permit us to obtain properties of the sub-Riemannian distance starting from those of the heat kernel and vice versa. 
For the Grushin plane endowed with the Euclidean volume we get the expansion $p_t(x,y)\sim t^{-5/4}\exp(-d^2(x,y)/4t)$ where $y$ is reached from a Riemannian point $x$ by a minimizing geodesic which is conjugate at $y$.
\end{abstract}

\section{Introduction}

The heat kernel on sub-Riemannian manifolds has been an object of
attention starting from the late 70s
\cite{mioheat,baudoin,garofalob,benarouscut,
bismut,brockettmansouri,gaveau,kusustroock,leandre,strichartz,taylor},
as have the geodesics and cut and conjugate loci of such manifolds
\cite{miosr3d,nostrolibro,bonnard-kupka,agrachevjp,agrexp,boscainrossi,gauthier-contact,yuri1,yuri234}.
 In this paper, we provide a general approach to relate the
sub-Riemannian distance to the small time asymptotics of the heat
kernel at the cut locus, at least in the case when there are no
abnormal minimizers to the relevant point in the cut locus.

The problem of relating the sub-Riemannian distance to the heat-kernel
 is an old problem (see for instance
\cite{hypoelliptic,mioheat,benarouscut,lanconellibook,boscain-laurent,boscainpolidoro,follandstein,jerisonsanchez,leandremaj,leandremin,montgomerybook,rothstein,sanchezcalle}).
In the following we recall some of the most relevant results.
Let $M$ be a $n$-dimensional smooth manifold provided with a complete 
sub-Riemannian structure, inducing a distance $d$, and also  provided with a smooth volume $\mu$. Let  $p_t(x,y)$
the heat kernel of the sub-Riemannian heat equation
$\partial_t\phi=\Delta \phi$, where $\Delta$ is the sub-Riemannian
Laplacian defined as the divergence of the horizontal gradient. In
particular $\Delta$ could be the sum of the squares of a choice of
vector fields defining the sub-Riemannian distance (possibly with a first order term belonging to the distribution).

\bi
\iii {\bf On the diagonal}. For some constant $C>0$ (depending on the
sub-Riemannian structure and $x$, we have
\bqn \label{eq:formula1}
p_t(x,x) =\frac{C+O(\sqrt{t})}{t^{Q/2}}.
\eqn
This result is due to Ben Arous and Leandre
\cite{benarousleandrediag}. Here $Q$ is the Hausdorff
dimension of the sub-Riemannian manifold at $x$ (see also
\cite{mioheat}).

\iii {\bf Off diagonal and off cut locus}.
Fix $x\neq y$. If $y$ is not in the cut locus of $x$ and 
there are no abnormals from $x$ to $y$, then for some constant $C>0$  (depending on the sub-Riemannian
structure, $x$, and  $y$), one has
\[
p_t(x,y) = \frac{C+O(t)}{t^{n/2}}  e^{-d^{2}(x,y)/4t}.
\]
This result is due to Ben Arous \cite{benarouscut}. See also Taylor \cite{taylor}.

\iii {\bf In any point of the space including the cut locus}.
 \bqn
 \label{eq-varadh}
\lim_{t\to0}4 t \log p_t(x,y)=-d^2(x,y).
\eqn
This result is due to Leandre \cite{leandremaj,leandremin} (see also Taylor
\cite{taylor}). It is very general 
but is rougher than the one of Ben Arous. Roughly speaking
it says that both on and off the cut locus, the leading term for
$t\to0$ has the form  $e^{-d^{2}(x,y)/4t}
$.

\ei
These results hold in particular in the Riemannian case. In that case
we have $Q=n$ and formula \r{eq-varadh} is the celebrated Varadhan
formula obtained in \cite{varadhan}.

\medskip
In this paper we give a finer result with respect to the one of
Leandre. We  show that if $y$ belongs to the cut locus of $x$ and all
minimizers connecting $x$ and $y$ are strongly normal (a
minimizer is said to be \emph{strongly normal} if every piece of it is
not abnormal) then  the rate of decay of $p_t(x,y)$ depends,
roughly, on ``how conjugate'' $x$ and $y$ are, along the minimal
geodesics connecting them.  Intuitively, the more conjugate they are,
the slower the decay.  These results include
Riemannian manifolds as a special case, for which they are
completely general, since there are no abnormal minimizers in
Riemannian geometry.  Some details of the explicit relationship between the heat
kernel asymptotics and the conjugacy of the minimal geodesics appears
to be new even in the Riemannian context. 
Our results are also completely general for certain
classes of sub-Riemannian geometries for which it is known there are
no abnormals, such as contact manifolds and CR-manifolds.
For a discussion of the presence of strictly abnormal minimizers in sub-Riemannian geometry one can see \cite{chitourjeantrelat}.

Our main result is Theorem \ref{t:mainmain} in Section \ref{s:conslapl}, which relates the heat
kernel asymptotics of $p_t(x,y)$ with what we call the \emph{hinged energy
function}
\bqn
\label{eq:hxy}
h_{x,y}(z)=\frac{1}{2}(d^{2}(x,z)+d^{2}(z,y)),\eqn on the set of
midpoints of all minimizing geodesics connecting $x$ to $y$. To avoid
overly complicated notation, we state here the following corollary
which explains what happens in the case when the first terms of the
Taylor expansion of $h_{x,y}$ have a simple expression.

\bc
 \label{t:main}
 Let $M$ be an $n$-dimensional complete sub-Riemannian manifold provided
with a smooth volume $\mu$ and let $p_{t}$ be the heat kernel of the
sub-Riemannian heat equation. Given distinct $x,y\in M$ let
$h_{x,y}(z)$ 
be the hinged energy
function. Assume that there is only one optimal geodesic joining $x$
to $y$ and that it is strongly normal, and let $z_{0}$ be the midpoint
of the geodesic.

 Then $h_{x,y}(z)$ is smooth in a  neighborhood of $z_{0}$ and attains
its minimum at $z_{0}$. Moreover, if there exists a coordinate system
$(z_{1},\ldots,z_{n})$ around $z_{0}$ such that we have the expansion
 $$h_{x,y}(z)=\frac 14
d^{2}(x,y)+z_{1}^{2m_{1}}+\ldots+z_{n}^{2m_{n}}+o(|z_{1}|^{2m}+\ldots+|z_{n}|^{2m_{n}}),$$
for some integers $1\leq m_1\leq m_2 \leq \cdots \leq m_n$
then for some constant $C>0$  (depending on the sub-Riemannian
structure, $x$, and  $y$), one has
$$p_{t}(x,y)=\frac{C+o(1)}{t^{n-\sum_{i}\frac{1}{2m_{i}}}}\exp\left(-\frac{d^{2}(x,y)}{4t}\right)
.
$$
\ec

From an analysis of the relation between the expansion of $h_{x,y}$
and the conjugacy of $x$ and $y$ we get
\bc
\label{corollarioMAIN}
Let $M$ be an $n$-dimensional complete sub-Riemannian manifold provided
with a smooth volume $\mu$ and let $p_{t}$ be the heat kernel of the
sub-Riemannian heat equation. Let $x$ and $y$ be distinct and assume that every optimal geodesic joining $x$
to $y$ is strongly normal. 

Then there exist positive
constants $C_i$, and $t_0$ (depending on
$M$, $x$, and $y$) such that
\begin{equation}\nn
(i)~~~~~~~~~~~~~~\frac{C_1}{t^{n/2}}  e^{-d^{2}(x,y)/4t} \leq p_t(x,y)
\leq\frac{C_2}{t^{n-(1/2)}}  e^{-d^{2}(x,y)/4t}, \quad \text{ for }
0<t<t_0.~~~~~~~~~~~~~~
\end{equation}
(ii) If $x$ and $y$ are conjugate along at least one minimal geodesic
connecting them, then
\[
p_t(x,y)\geq \frac{C_3}{t^{(n/2)+(1/4)}}  e^{-d^{2}(x,y)/4t}, \qquad
\text{ for } 0<t<t_0.
\]
(iii) If $x$ and $y$ are not conjugate along any
minimal geodesic joining them, then
\[
p_t(x,y) = \frac{C_4+O(t)}{t^{n/2}}  e^{-d^{2}(x,y)/4t}, \qquad \text{
for } 0<t<t_0.
\]
\ec
Note that (iii) shows that the result of Ben Arous \eqref{eq:formula1}
holds not only off the cut locus, but also on the cut locus if $x$ and
$y$ are not conjugate.

In the corollaries above the concept of sub-Riemannian manifold is
quite general. It includes Riemannian manifolds and even
sub-Rieman\-nian manifolds which are rank-varying (see Sections
\ref{s:srg} and Appendix \ref{appendix} for the precise definition).
The estimates (i) and (iii) were already known in Riemannian geometry (see \cite{Hsu} and \cite{Molchanov} respectively) while (ii) appears to be new even in the
Riemannian context.

The sub-Riemannian heat equation is intended with respect to the
sub-Riemannian Laplacian which is defined as the divergence of the
sub-Riemannian gradient. Here the divergence is computed with respect
to a smooth volume. In the equiregular case (see Definition
\ref{d:equiregular}) the most natural volume is Popp's volume,
introduced by Montgomery in his book \cite{montgomerybook}.   The
hypothesis that the sub-Laplacian is computed with respect to a smooth
volume is also essential. For rank-varying sub-Riemannian structures
or for sub-Riemannian structure which are not equiregular, one could
be tempted to define a sub-Laplacian containing diverging terms with
the Popp volume (which is also diverging). This approach is possible.
However it provides completely different results with respect to those
presented in this paper. See for instance \cite{boscain-laurent} for this approach in 
the case of the Grushin and Martinet structures.

In addition to these general bounds on the decay of $p_t(x,y)$ our
approach provides a technique for computing the heat kernel
asymptotics in concrete situations, subject, of course, to one's
ability to determine explicit information about the minimal geodesics
from $x$ to $y$ and the behavior of $h_{x,y}$ near their midpoints
(which is related to the conjugacy of the minimal geodesics).

Conversely, these results allow us to realize the old idea of getting
properties of the sub-Riemannian distance from those of the heat
kernel (see for instance \cite{hypoelliptic,montgomerybook}).  

The hypothesis that all optimal geodesics connecting $x$ to $y$ are
strongly normal is essential. In the case in which $x$ is reached by
$y$ along an abnormal minimizer, it is not clear how to measure 
how much $y$ is conjugate to $x$ since abnormal extremals are not included
in the exponential mapping and are in a sense isolated. The analysis
of the heat kernel asymptotics in the presence of abnormal minimizers
is an extremely difficult problem (also because of the lack of
information about properties of the sub-Riemannian distance) and its
study goes beyond the purpose of this paper.\\

\emph{Remark.} Notice that in our approach we start from the
sub-Rieman\-nian structure $(M,\distr,\metr)$, then we define an
intrinsic volume, and finally we build the Laplace operator naturally
associated with these data. This operator is by construction
symmetric, negative and has the form
$\lapl=\sum_{i=1}^k X_i^2+X_0$ where $\{X_i\}_{i=1}^k$ define an
orthonormal frame  satisfying the H\"omander condition and
$X_0\in\mbox{span} \{X_i\}_{i=1}^k$.

In the literature one more often finds the reverse procedure
\cite{jerisonsanchezcalleTT,rothstein} (see also  \cite{garofalob} and
references therein). One starts from a second order differential
operator with smooth coefficients ${\mathcal L}$ which is symmetric
and negative with respect to  a volume $\mu$, and then looks for a
distance as a function of which one can give estimates of the
fundamental solution of $\partial_t-{\mathcal L}$. This distance is
constructed by introducing the so-called {\sl sub-unit} curves for the
operator, see for instance (\cite{garofalob,lanconellibook}). When
${\mathcal L}$ is of the form ${\mathcal L}=\sum_{i=1}^k X_i^2+X_0$
where $\{X_i\}_{i=1}^k$ are linearly independent vector fields
satisfying the H\"omander condition, the symmetry with respect to
$\mu$ implies that $X_0\in\mbox{span} \{X_i\}_{i=1}^k$. Moreover the
distance one gets is the sub-Riemannian distance for which
$\{X_i\}_{i=1}^k$ is an orthonormal frame.

Also, let us mention that a wide literature is available about
operators of the type ${\mathcal L}=\sum_{i=1}^k X_i^2+X_0$ where
$\{X_i\}_{i=1}^k$ satisfy the H\"omander condition, but
$X_0\notin\mbox{span}\{X_i\}_{i=1}^k$.

\subsection{Structure of the paper}
 The structure of the paper is as follows.  In Section \ref{s:srg} we
introduce the concept of sub-Riemannian manifold. To avoid heavy
notation, we have decided to restrict ourselves to  the case in which
the dimension of the distribution does not depend on on the point. The
rank-varying case is postponed to Appendix \ref{appendix}. All the
results of the paper holds also in this case.

In Sections \ref{s:laplace1} we state and prove a result expressing the heat kernel asymptotic as a Laplace integral over a neighborhood of the set of midpoints of minimal geodesics (see Theorem \ref{THM:MainExpansion}). In Section \ref{s:laplace} we discuss the asymptotics of Laplace type integrals and we discuss the relation between the degeneracy of the hinged energy function $h_{x,y}$ around the midpoints the minimal geodesics connecting $x$ and $y$ and the conjugacy of the minimal geodesics connecting them (see Theorem \ref{THM:HAndConj}). Then in Section \ref{s:conslapl}, we get our main general result, namely the estimates on the heat kernel $p_t(x,y)$
as a consequence of the previous analysis (see Theorem \ref{t:mainmain}).

In Section \ref{s:examples} we apply our general results to some
relevant cases. We briefly illustrate our results on the Heisenberg
group for which both the optimal synthesis (i.e.\ the set of all
optimal trajectories) from a given point and the heat kernel are
known.

The second example is  the nilpotent free $(3,6)$ case. In this case we get an asymptotic expansion on the vertical subspace (see Section \ref{s:36}), where all points are conjugate along minimal geodesics, which agrees with the fact that there exists a one parameter family of optimal geodesic reaching these points.

Finally in Section \ref{s:gru} we study the heat kernel in the Grushin plane, with respect to the standard Lebesgue measure. The Grushin structure is the rank-varying sub-Riemannian structure on the plane $(x,y)$ such that $X=\partial_x$ and $Y=x\partial_y$ define an orthonormal frame. The corresponding sub-Laplacian is $\lapl=X^2+Y^2$. Starting from the Riemannian point $q_0=(-1,-\pi/4)$ we get, for the asymptotic at the point $q_1=(1,\pi/4)$, which is reached from $q_0$ by a minimizing geodesic which is conjugate at $q_1$, the expression $p_t(q_0,q_1)\sim t^{-5/4}\exp(-d^2(q_0,q_1)^2/4t)$, computing explicitly the degeneration of the hinged energy function. To our knowledge this is the first time in which an expansion of the type $t^{-\alpha} \exp(-d^2(q_0,q_1)/4t)$, with $\alpha \neq N/2$ for an integer $N$, is observed in the Riemannian or sub-Riemannian context.

\section{Sub-Riemannian geometry}\label{s:srg}
We start by recalling the definition of sub-Riemannian manifold in the
case of a distribution of constant rank $k$ smaller than the dimension of the space. 
For the more general definition of rank-varying sub-Riemannian structure (including as a particular case Riemannian structures) see Appendix
\ref{appendix}.
\bdeff
A \emph{sub-Riemannian manifold} is a triple $(M,\distr,\metr)$,
where
\bi
\iii[$(i)$] $M$ is a connected orientable smooth manifold of dimension
$n\geq 3$;
\iii[$(ii)$] $\distr$ is a smooth distribution of constant rank $k< n$
satisfying the \emph{H\"ormander condition}, i.e.\ a smooth map that
associates to $q\in M$  a $k$-dimensional subspace $\distr_{q}$ of
$T_qM$ such that
\bqn \label{Hor}
\qquad \text{span}\{[X_1,[\ldots[X_{j-1},X_j]]]_{q}~|~X_i\in\overline{\distr},\,
j\in \N\}=T_qM, \ \all q\in M,
\eqn
where $\overline{\distr}$ denotes the set of \emph{horizontal smooth
vector fields} on $M$, i.e.\
$$\overline{\distr}=\Pg{X\in\mathrm{Vec}(M)\ |\ X(q)\in\distr_{q}~\
\forall~q\in M}.$$
\iii[$(iii)$] $\metr_q$ is a Riemannian metric on $\distr_{q}$ which is smooth
as function of $q$. We denote  the norm of a vector $v\in \distr_{q}$
by 
$|v|_{\metr}=\sqrt{\metr_{q}(v,v)}.$
\ei
\edeff

A Lipschitz continuous curve $\g:[0,T]\to M$ is said to be
\emph{horizontal} (or \emph{admissible}) if
$$\dot\g(t)\in\distr_{\g(t)}\qquad \text{ for a.e.\ } t\in[0,T].$$

Given an horizontal curve $\g:[0,T]\to M$, the {\it length of $\g$} is
\bqn
\label{e-lunghezza}
\ell(\g)=\int_0^T |\dot{\g}(t)|_{\metr}~dt.
\eqn
Notice that $\ell(\g)$ is invariant under time reparametrization of
the curve $\g$.
The {\it distance} induced by the sub-Riemannian structure on $M$ is the
function
\bqn
\label{e-dipoi}
d(q_0,q_1)=\inf \{\ell(\g)\mid \g(0)=q_0,\g(T)=q_1, \g\ \mathrm{horizontal}\}.
\eqn
The \hp\ of connectedness of $M$ and the H\"ormander condition
guarantees the finiteness and the continuity of $d(\cdot,\cdot)$ with
respect to the topology of $M$ (Chow-Rashevsky theorem, see for
instance \cite{agrachevbook}). The function $d(\cdot,\cdot)$ is called
the \emph{Carnot-Caratheodory distance} and gives to $M$ the structure
of a metric space (see \cite{agrachevbook}).

Locally, the pair $(\distr,\metr)$ can be given by assigning a set of
$k$ smooth vector fields spanning $\distr$ and that are orthonormal
for $\metr$, i.e.\
\bqn
\label{trivializable}
\distr_{q}=\text{span}\{X_1(q),\dots,X_k(q)\}, \qquad
\metr_q(X_i(q),X_j(q))=\delta_{ij}.
\eqn
In this case, the set $\Pg{X_1,\ldots,X_k}$ is called a \emph{local
orthonormal frame} for the sub-Riemannian structure.

The sub-Riemannian metric can also be expressed locally in ``control
form'' as follows. We consider the control system,
\bqn\label{eq:lnonce0}
\dot q=\sum_{i=1}^m u_i X_i(q)\,,~~~u_i\in\R\,,
\eqn
and the problem of finding the shortest curve that joins two fixed
points $q_0,~q_1\in M$ is naturally formulated as the optimal control
problem
\bqn \label{eq:lnonce}
~~~\int_0^T  \sqrt{
\sum_{i=1}^m u_i^2(t)}~dt\to\min, \ \ \qquad q(0)=q_0,~~~q(T)=q_1\neq q_{0}.
\eqn

\bdeff\label{d:equiregular}
Define $\distr^{1}:=\distr,
\distr^{i+1}:=\distr^{i}+[\distr^{i},\distr]$, for every $i\geq1$. A
sub-Riemannian manifold is said to be \emph{equiregular} if for each
$i\geq1$, the dimension of $\distr^{i}_{q}$ does not depend on the
point $q\in M$. For an equiregular sub-Riemannian manifold the
H\"ormander condition guarantees that there exists (a minimal) $m\in
\N$, called \emph{step} of the structure, such that
$\distr_{q}^{m}=T_{q}M$, for all $q\in M$. The sequence
$$\mc{G}:=
(\underset{\begin{smallmatrix} \shortparallel \\ k
\end{smallmatrix}
}{\text{dim}\,\distr},\text{dim}\,\distr^{2},\ldots,\underset{\begin{smallmatrix}
\shortparallel \\ n
\end{smallmatrix}
}{\text{dim}\, \distr^{m}}),$$
is called the \emph{growth vector} of the sub-Riemannian manifold. The
growth vector permits us to compute the Hausdorff dimension of $(M,d)$
as a metric space (see \cite{mitchell})
\bqn \label{eq:Q}
Q=\sum_{i=1}^m i k_i, \qquad k_i:=\dim \distr^i - \dim \distr^{i-1}.
\eqn
In particular the Hausdorff dimension is always bigger than the
topological dimension of $M$.
\edeff

\bdeff \label{d:nilp}
A sub-Riemannian manifold is said to be \emph{nilpotent} if $M$ is a
nilpotent Lie group and the sub-Riemannian structure is left-invariant
with respect to the group operation.
\edeff

\subsection{Minimizers and geodesics}
 In this section we briefly recall some facts about sub-Riemannian
geodesics. In particular, we define the
sub-Riemannian exponential map.

 \bdeff A \emph{geodesic} for a sub-Riemannian manifold
$(M,\distr,\metr)$ is an admissible curve $\g:[0,T]\to M$ such that
 $|\dot{\g}(t)|_{\metr}$  is constant and, for every sufficiently
small interval $[t_1,t_2]\subset [0,T]$, the restriction
$\g_{|_{[t_1,t_2]}}$ is a minimizer of $\ell(\cdot)$.
A geodesic for which $|\dot{\g}(t)|_{\metr}=1$ is said to be
parameterized by arclength.

A sub-Riemannian manifold is said to be \emph{complete}
if $(M,d)$ is complete as a metric space.
If the sub-Riemannian
metric is the restriction to $\distr$ of a complete Riemannian metric,
then it is complete.

Under the assumption that the manifold is complete, a
version of the Hopf-Rinow theorem (see \cite[Chapter 2]{burago})
implies that the manifold is geodesically complete (i.e. all geodesics are defined for every $t\geq0$) and that for every two points there exists a minimizing geodesic
connecting them.
\edeff
Trajectories minimizing the distance between two points are solutions
of first order necessary conditions for optimality, which in the case
of sub-Riemannian geometry are given by a weak version of the
Pontryagin Maximum Principle (\cite{pontrybook}).
\bt\label{t:pmpw}
Let $q(\cdot):t\in[0,T]\mapsto q(t)\in M$ be a solution of the
minimization problem \eqref{eq:lnonce0},\eqref{eq:lnonce} such that
$|\dot q(t)|_{\metr}$  is constant and $u(\cdot)$ be the
corresponding control.
Then there exists a Lipschitz map $p(\cdot): t\in [0,T] \mapsto
p(t)\in T^{*}_{q(t)}M\setminus\{0\}$  such that one and only one of
the following conditions holds:
\bi
\iii[(i)]
$
\dot{q}=\dfrac{\partial H}{\partial p}, \quad
\dot{p}=-\dfrac{\partial H}{\partial q}, \quad
u_{i}(t)=\la p(t), X_{i}(q(t))\ra,
\\$
where $H(q,p)=\frac{1}{2} \sum_{i=1}^{k} \la p,X_{i}(q)\ra^{2}$.
\vspace{0.2cm}
\iii[(ii)]
$
\dot{q}=\dfrac{\partial \HH}{\partial p}, \quad
\dot{p}=-\dfrac{\partial \HH}{\partial q}, \quad
0=\la p(t), X_{i}(q(t))\ra,
\\$
where $\HH(t,q,p)=\sum_{i=1}^{k}u_{i}(t) \la p,X_{i}(q)\ra$.
\ei
\et
\noindent
For an elementary proof of Theorem \ref{t:pmpw} see \cite{nostrolibro}.

\brem If $(q(\cdot),p(\cdot))$ is a solution of (i) (resp.\ (ii)) then
it is called a \emph{normal extremal} (resp.\ \emph{abnormal
extremal}). It is well known that if $(q(\cdot),p(\cdot))$ is a normal
extremal then $q(\cdot)$ is a geodesic (see
\cite{nostrolibro,agrachevbook}). This does not hold in general for
abnormal extremals. An admissible trajectory $q(\cdot)$ can be at the
same time normal and abnormal (corresponding to different covectors).
If an admissible trajectory $q(\cdot)$ is normal but not abnormal, we
say that it is \emph{strictly normal}.

Abnormal extremals are very difficult to treat and many questions are
still open. For instance it is not known if abnormal minimizers are
smooth (see \cite{montgomerybook}).
\erem

\bdeff
A minimizer $\gamma:[0,T]\to M$ is said to be \emph{strongly normal}
if for every $[t_1,t_2]\subset [0,T]$, $\gamma|_{[t_1,t_2]}$ is not an
abnormal minimizer.\edeff

In the following we denote by $(q(t),p(t))=e^{t\vec{H}}(q_{0},p_{0})$
the solution of $(i)$ with initial condition
$(q(0),p(0))=(q_{0},p_{0})$. Moreover we denote by $\pi:T^{*}M\to M$
the canonical projection.

 Normal extremals (starting from $q_{0}$) parametrized by arclength
correspond to initial covectors $p_{0}\in \Lambda_{q_{0}}:=\{p_{0}\in
T^{*}_{q_{0}}M | \, H(q_{0},p_{0})=1/2\}.$
\bdeff  \label{d:exp}
Let $(M,\distr,\metr)$ be a complete sub-Riemannian manifold and
$q_0\in M$. We define the \emph{exponential map} starting from $q_{0}$
as
\bqn \label{eq:expmap}
\EXP_{q_{0}}: \Lambda_{q_{0}}\times \R^{+} \to M, \qquad
\EXP_{q_{0}}(p_{0},t)= \pi(e^{t\vec{H}}(q_{0},p_{0})).
\eqn
\edeff

Next, we recall the definition of cut and conjugate time.

\bdeff \label{def:cut} Let $q_{0}\in M$ and
$\g(t)$ an arclength geodesic starting from $q_{0}$.
The \emph{cut time} for $\g$ is $\tcut(\g)=\sup\{t>0,\, \g|_{[0,t]}
\text{ is optimal}\}$. The \emph{cut locus} from $q_{0}$ is the set
$\Cut(q_{0})=\{\g(\tcut(\g)), \g$  arclength geodesic from
$q_{0}\}$.
\edeff

\bdeff \label{def:con} Let $q_{0}\in M$ and
$\g(t)$ a normal arclength geodesic starting from $q_{0}$ with initial
covector $p_{0}$. Assume that $\g$ is not abnormal.
The \emph{first conjugate time} of $\g$ is
$\tcon(\g)=\min\{t>0,\  (p_{0},t)$ is a critical point of
$\EXP_{q_{0}}\}$. The (first) \emph{conjugate locus} from $q_{0}$ is the
set $\Con(q_{0})=\{\g(\tcon(\g)), \g$  arclength geodesic from
$q_{0}\}$.

\edeff
It is well known that, for a geodesic $\g$ which is not abnormal, the
cut time $t_{*}=\tcut(\g)$ is either equal to the conjugate time or
there exists another geodesic $\til{\g}$ such that
$\g(t_{*})=\til{\g}(t_{*})$ (see for instance \cite{agrexp}).

\brem
In sub-Riemannian geometry, the exponential map starting from $q_{0}$ is never a local
diffeomorphism in a neighborhood of the point $q_{0}$ itself. As a consequence the
sub-Riemannian balls are never smooth and both the cut and the
conjugate loci from $q_{0}$ are adjacent to the point $q_{0}$ itself
(see \cite{agratorino}).
\erem

\subsection{The sub-Laplacian}\label{s:lapl}
In this section we define the sub-Riemannian Laplacian
 on a sub-Riemannian manifold $(M, \distr, \metr)$, provided with a
smooth volume $\mu$.

The sub-Laplacian is the natural generalization of the
Laplace-Beltrami operator defined on a Riemannian manifold, defined as
the divergence of the gradient.

The sub-Riemannian gradient can be defined with no
difficulty.
On a sub-Riemannian manifold $(M,\distr,\metr)$, the gradient is the
unique operator $\grad: \mathcal{C}^\infty(M)\to \overline{\distr}$ defined by $$
\metr_q(\grad\phi(q),v)=d\phi_{q}(v), \quad \all \phi \in
\mc{C}^{\infty}(M), \, q\in M,~v\in \distr_q.$$
By definition, the gradient is a horizontal vector field. If
$X_{1},\ldots,X_{k}$ is a local orthonormal frame, it is easy to see
that it is written as follows
$\grad \phi= \sum_{i=1}^{k} X_i(\phi) X_i,$
where $X_i(\phi)$ denotes the Lie derivative of $\phi$ in the
direction of $X_i$.

The divergence of a vector field $X$ with respect to a volume $\mu$ is
the function $\dive\, X$ defined by the identity $L_{X}\mu=(\dive\, X
)\mu$, where $L_{X}$ stands for the Lie derivative with respect to $X$.

The sub-Laplacian associated with  the sub-Riemannian structure, i.e.\
$\lapl \phi=\dive(\grad \phi),$
 is written in a local orthonormal frame $X_{1},\ldots,X_{k}$ as follows
\bqn \label{eq:lapldiv}
\lapl= \sum_{i=1}^{k}X_{i}^{2}+ (\dive\, X_{i}) X_{i}.
\eqn
Notice that $\lapl$ is always expressed as the sum of squares of the
element of the orthonormal frame plus a first order term that belongs
to the distribution and depends on the choice of the volume $\mu$.

The existence of a smooth heat kernel for the operator \eqref{eq:lapldiv}, in the case of a  complete sub-Riemannian manifold, is stated in \cite{strichartz}.
 
\subsubsection{Popp's volume and the instrinsic sub-Laplacian}\label{popp}

In this section we recall how to construct an intrinsic Laplacian
(i.e.\ that depends only on the sub-Riemannian structure) in the case
of an equiregular sub-Riemannian manifold.

On a Riemannian manifold the Euclidean structure defined on the
tangent space defines in a standard way a canonical volume: the
Riemannian volume.

In the case of an equiregular sub-Riemannian manifold
$(M,\distr,\metr)$, even if there is no global scalar product defined
in $T_{q}M$, it is possible to define an intrinsic volume, namely the
Popp volume \cite{montgomerybook}. This is a smooth volume on $M$ that
is defined from the properties of the Lie algebra generated by the
family of the horizontal vector fields. In the Riemannian case this
coincide with the Riemannian volume.

On an equiregular manifold of dimension 3 the Popp volume is easily
defined as $\nu_1 \wedge \nu_2 \wedge \nu_3$, where
$\nu_1,\nu_2,\nu_3$ is the dual basis to $X_{1},X_{2}$ and
$[X_{1},X_{2}]$, where $\{X_{1},X_{2}\}$ is any local orthonormal
frame for the structure. This definition happens to be independent on
the choice of $X_1,X_2$. For the general definition see, e.g
\cite{hypoelliptic,corank1}.

Notice that the Popp volume is not the unique intrinsic volume that
one can build from the geometric structure of $(M,\distr,\metr)$.
Since a sub-Riemannian manifold is a metric space (with the
Carnot-Caratheodory distance), one can define the $Q$-dimensional
Hausdorff measure on $M$, where $Q$ is defined in \eqref{eq:Q}. In
contrast with the Riemannian case, starting from dimension 5, the
$Q$-dimensional Hausdorff measure the does not coincide in general
with Popp's  (see \cite{corank1,corank2} for details about these
results).

The intrinsic sub-Laplacian is defined as the sub-Laplacian where the
divergence is computed  with respect to the Popp volume.

\brem \label{r:group}
In the case of a left-invariant structure on a Lie group (and in
particular for a nilpotent structure), the Popp volume is
left-invariant, hence proportional to the left Haar measure.

For unimodular Lie groups, and in particular for nilpotent groups, one
gets for the intrinsic sub-Laplacian  the \virg{sum of squares} form
(see \cite{hypoelliptic})
$$\lapl=\sum_{i=1}^k X_i^2.$$
\erem

\brem \label{r:rv}
To define the Popp volume the equiregularity assumption is crucial. In
the non-equiregular case or in the rank-varying case the Popp volume
diverges approaching the non-regular points, as do the coefficients of the intrinsic
sub-Laplacian \cite{hypoelliptic,boscain-laurent}.
\erem

\section{General expression as a Laplace integral}
\label{s:laplace1}

From now on by a sub-Riemannian manifold we mean a structure in the sense of Section \ref{s:srg} or Appendix \ref{appendix}, which include as a particular case Riemannian structures. 

In the following we denote by $\Sigma \subset M\times M$ the set of
pairs $(x,y)$ with $x\neq y$ such that there exists a unique
minimizing
geodesic from $x$ to $y$ and such that this geodesic is strictly
normal and not conjugate. Notice that $\Sigma$ is an open set in
$M\times M$ (see \cite{agrachevsmooth,riffordtrelat} and \cite[Chapter: Regularity of SR
distance]{nostrolibro}).

Recall that the heat kernel is the the fundamental solution of the
heat equation  $\partial_t\phi=\Delta \phi$, where $\Delta$ is the
sub-Riemannian Laplacian defined with respect to some smooth volume
$\mu$ on a sub-Riemannian manifold
$M$.
We begin by recalling the asymptotic expansion of the heat kernel
away from the cut locus, due to Ben Arous \cite{benarouscut} (see
Theorem 3.1, and adjust for the fact that our heat kernel is for
``$\Lap$'' rather than ``$\Lap/2$'').
\bt\label{t:BA}
Let $M$ be an $n$-dimensional complete sub-Riemannian
manifold in the sense of Section \ref{s:srg} or Appendix \ref{appendix}, with a smooth volume $\mu$ and associated heat kernel $p_t$ and
let $(x,y)\in \Sigma$.
Then for every non-negative integer $m$, we have the following
asymptotic expansion as $t\da 0$:
\bqn\label{eq:BA}
p_{t}(x,y)= \frac{1}{t^{n/2}}\exp\left(-\frac{d^{2}(x,y)}{4t}\right)\left(\sum_{j=0}^{m}c_{j}(x,y)t^{j}+O(t^{m+1})\right).
\nn
\eqn
Here the $c_i$ are smooth functions on $\Sigma$ with $c_0(x,y)>0$.  Further, if
$K\subset \Sigma$ is a compact set, then the expansion is uniform over $K$.
\et

We will also need some preliminary control of the heat kernel at the
cut locus, which is provided by a well-known result of Leandre
\cite{leandre}.  In particular, Theorem 1 of \cite{leandremaj} and
Theorem 2.3 of \cite{leandremin} give (again taking into account our
normalization of the heat kernel)

\bt\label{THM:Leandre}
Let $M$ be a  complete sub-Riemannian
manifold with a
smooth volume $\mu$ and associated heat kernel $p_t$.  For any
compact subset $K$ of $M\times M$, the following holds uniformly for
$(x,y)\in K$:
\[
\lim_{t\searrow 0} 4t \log p_t(x,y) = -d^2(x,y) .
\]
\et
\brem Theorem \ref{t:BA} and \ref{THM:Leandre} were originally stated in $\R^n$ for sub-Riemannian metrics whose orthonormal frame consists of vector fields which are bounded with bounded derivatives. However, it is not hard to see that these results hold in the more general context of  complete sub-Riemannian structures (where closed balls are compact).
\erem
{\bf Notation}.  
In what follows we use sometimes the abbreviation $E(x,y)$ $=d^2(x,y)/2$ for the energy function.  For any two
distinct points $x$ and $y$, we let $\Gamma$ be the set of midpoints
of minimal geodesics from $x$ to $y$.  Further, we let $\GE$ be a
neighborhood of $\Gamma$, which we will feel free to make small enough
to satisfy various assumptions.  Finally, we
let $h_{x,y}(z) = E(x,z)+E(z,y)$ be the hinged energy function.  It's
clear from the definition that $h_{x,y}(z)$ is continuous.
\bl \label{Lem:AboutGamma}
The function $h_{x,y}$ attains its minimum exactly on $\Gamma$ and
$\min h_{x,y}= d^{2}(x,y)/4$.
\el
\begin{proof}
Let us consider a geodesic joining $x$ and $y$ and denote its midpoint by $z_0$.
We want to prove that $h_{x,y}(z_0)\leq h_{x,y}(z)$ for every $z$
and that we have equality if and only if $z$ is a midpoint of a
geodesic joining $x$ and $y$.

Let $a=d(x, z_0)=d(z_0,y)$, $b=d(x,z)$, and $c=d(y,z)$.
By the triangle inequality, we have $2a\leq b+c$. Moreover we can
assume that both $b$ and $c$ are less than or equal to $2a$, since
otherwise the statement is trivial.  Let $\eps \geq 0$ be such that
$2a+\eps=b+c$ and compute
\begin{align*}
h_{x,y}(z) &=\frac{1}{2}(b^{2}+c^{2})
 =\frac{1}{2}((2a+\eps-c)^{2}+c^{2})\\
&\geq a^{2}+(a-c)^{2}+\frac{\eps^{2}}{2}+\eps(2a-c)
\geq a^{2}=h_{x,y}(z_0)
\end{align*}
Moreover we have equality in the two inequalities if and only if
$\eps=0$ and $a=c$, which is precisely the case where $z$ is the
midpoint of a geodesic
joining  $x$ and $y$. Finally $h_{x,y}(z_0)=d^{2}(x,y)/4$.
\end{proof}

We will need some basic assumptions about, and properties of, $\GE$.
First, basic properties of the distance function on $M$ imply that
$\Gamma$ is compact.  Next, all of our work will take place under the
condition that we are ``away from'' any abnormal geodesics.  In
particular, assume that $x$ and $y$ are distinct and that every
minimizer from $x$ to $y$ is a strongly normal geodesic.  While we
certainly allow $y$ to be in the cut locus of $x$ (which is a
symmetric arrangement),
the midpoint of every minimal geodesic from $x$ to $y$ will be a
positive distance from the cut loci of both $x$ and $y$.

More precisely, let $\lam\in T^{*}_xM$ be a covector  such that
$\Exp_x(t\lam)$ for $t\in[0,d(x,y)]$ parametrizes a minimal geodesic from
$x$ to $y$. 
(Here we adopt the convention that
$\Exp_x(\lambda)=\pi\circ e^{\vec{H}}(x,\lambda)$.)
 Call this geodesic $\gamma$, and let $z_0$ be its
midpoint.  Since the cut time along
$\gamma$ is at least $d(x,y)$ and since the cut time is continuous as
a function on $T^{*}_xM$ near $\lam$, it follows that $\Exp_x$ is a
diffeomorphism from a neighborhood $U$ of $(d(x,y)/2)\lam$ to a
neighborhood $U^{\prime}$ of $z_0=\Exp_x\lp (d(x,y)/2)\lam\rp$.  Further,
assuming $U$ small
enough, there is a unique minimal geodesic from $x$ to each point
$\Exp_x(\xi)$ in $U^{\prime}$ given by $\Exp_x(s\xi)$ for $s\in[0,1]$, and
this geodesic is not conjugate.  In the case when $\xi=(d(x,y)/2)\lam$, we
have the ``first half'' of $\gamma$, from $x$ to $z_0$.  Because
$\gamma$ is strongly normal, this piece of $\gamma$ is strictly
normal.  Because the property of being strictly normal is an open
condition on geodesics (see \cite{agrachevsmooth,riffordtrelat} and \cite[Chapter: Regularity of SR
distance]{nostrolibro}), after possibly
shrinking $U$ and $U^{\prime}$ we have that all  of the minimal
geodesics from $x$ to points in $U^{\prime}$ are also strictly normal.

One consequence is that by choosing $U$ (and thus $U^{\prime}$) small
enough, the distance function from $x$ is smooth on
$U^{\prime}$ (which we recall is some neighborhood of $z_0$, the midpoint of
a minimal geodesic $\gamma$ from $x$ to $y$).  Another is that (after
possibly further shrinking $U$ and $U^{\prime}$) the Ben Arous expansion holds
for $p_t(x,z)$ uniformly for $z\in U^{\prime}$.

Note that the discussion in the previous paragraph also holds if we
reverse the roles of $x$ and $y$.  Then,
since $\Gamma$ is compact, we see that for sufficiently small neighborhood $\GE$
the distance functions from both $x$ and $y$ are smooth on $\GE$ and
the Ben Arous expansion holds for both $p_t(x,z)$ and $p_t(y,z)$
uniformly for $z\in \GE$.  It follows that $h_{x,y}$ is also smooth on
$\GE$.  These are the key consequences of assuming that every
minimizer from $x$ to $y$ is a strongly normal geodesic.  We will also
occasionally take
advantage of the structure of the exponential map based at either $x$
or $y$ in a neighborhood of any point $z\in \GE$.  From now on, we
will assume that, for such $x$ and $y$, $\GE$ is chosen in this way.

We now describe the main idea for determining the expansion on the cut
locus.  The intuition benefits from recalling that the heat kernel is
also the transition density of Brownian motion on $M$.  By the
semi-group property (or the Markov property, from a stochastic point
of view), a particle that travels from a point $x$ to a point $y\neq
x$ in time $t$ first goes to some ``halfway'' point at time $t/2$, and
then continues the rest of the way to $y$.  For small $t$, a particle
traveling from $x$ to $y$ is most likely to do so via a path which is
approximately a geodesic (traversed at uniform speed).  This is the
usual intuition from large deviation theory.  Thus, at time $t/2$,
such a particle is likely to be near the midpoint of some minimal
geodesic from $x$ to $y$.  The key insight, originally due to
Molchanov \cite{Molchanov} in the Riemannian case, is that, even in the case
$y\in\Cut(x)$, we can choose $\GE$ as just discussed so that the expansion of
Ben Arous can be applied to both the first and second halves of the
particle's journey from $x$ to $y$ (at least with high probability).
The expansion at the cut locus is thus obtained by ``gluing together''
two copies of the Ben Arous expansion along the midpoints of the
minimal geodesics form $x$ to $y$.  Making this argument precise,
using only geometric analysis (we use stochastic notions
only to bolster our intuition in the present paper), provides the
proof of the next theorem. (The same \virg{gluing idea} was employed to compute asymptotics of logarithmic derivatives in the Riemannian case in \cite{Neel,neelstroock}.)

\begin{THM}\label{THM:MainExpansion}
Let $M$ be an $n$-dimensional  complete sub-Riemannian
manifold in the sense of Section \ref{s:srg} or Appendix \ref{appendix}, and let $x$
and $y$ be distinct points such that all minimal geodesics from $x$ to
$y$ are strongly normal.  Then for $\GE$ and $c_0$ as above there exists
$\delta>0$ such that
\begin{align*}
p_t(x,y) = \int_{\GE} \frac{2^n}{t^n} &e^{-h_{x,y}(z)/t}\lp
c_0(x,z)c_0(z,y)+O(t)\rp \, \mu(dz)\\
 &+ o\lp \exp\lb \frac{-E(x,y)/2 -\delta}{t} \rb \rp .
\end{align*}
Here the $O(t)$ term in the integral is uniform over $\GE$.
\end{THM}

\emph{Proof:} By the semi-group property (or Chapman-Kolmogorov
equation, for probabilists), we have
\[
p_t(x,y) = \int_{M} p_{t/2}(x,z) p_{t/2}(z,y) \, \mu(dz) .
\]
We first divide $M$ into two regions, $\GE$ and $M\setminus \GE$.  As
just discussed, both
$p_t(x,\cdot)$ and $p_t(\cdot,y)$ are uniformly approximated by the
Ben Arous expansion on $\GE$ (since we assume that $\eps>0$ is
sufficiently small).
Using just the first term, we see that
\begin{align*}
p_t(x,y) = \int_{\GE} &\lp\frac2t\rp^{n} e^{-h_{x,y}(z)/t}\lp c_0(x,z)c_0(z,y)+O(t)\rp \,
\mu(dz)\\
 &+ \int_{M\setminus\GE} p_{t/2}(x,z) p_{t/2}(z,y) \, \mu(dz) ,
\end{align*}
where the $O(t)$ terms are uniform over $\GE$ by the uniformity of the
Ben Arous expansion there.

Next, we estimate the integral over $M\setminus\GE$.  First, assume
that $M$ is compact.  By Theorem \ref{THM:Leandre}, we have that, on
$M$,
\[
p_t(u,v) = \exp\lb \frac{-d^2(u,v)/2 + r(t,u,v)}{2t}\rb,
\]
where $r(t,u,v)$ goes to zero uniformly with $t$ on all of $M$.  (In
the remainder of the proof, we will use $r$ to denote a function with
this property, the exact definition of which may change from line to
line.)  We see that
\[
p_{t/2}(x,z) p_{t/2}(z,y) = \exp\lb \frac{-h_{x,y}(z) + r(t,z)}{t}\rb .
\]
Further, the minimum of $h_{x,y}(z)$ on $M\setminus\GE$ is strictly
greater than $h_{x,y}(\Gamma)=E(x,y)/2$.  Because $M\setminus\GE$ has
finite volume (by compactness), we see that there exists $\delta>0$
such that
\begin{equation}\label{Eqn:OffGamma}
\int_{M\setminus\GE} p_{t/2}(x,z) p_{t/2}(z,y) \, \mu(dz) =o\lp
\exp\lb \frac{-E(x,y)/2 -\delta}{t} \rb \rp .
\end{equation}

Next, consider the case when $M$ is not compact.  Then for large
enough $R$, we see that $x$, $y$, and $\GE$ are all inside of $B_x(R)$
(the ball of radius $R$ centered around $x$).  We split the intergal
over $M\setminus\GE$ into an integral over $B_R(x)\setminus\GE$ and an
integral over $M\setminus B_R(x)$.  The previous argument can be
applied to the integral over $B_R(x)\setminus\GE$.  Further, for large
enough $R$, the integral over
$M\setminus B_R(x)$ is also $o\lp \exp\lb (-E(x,y)/2 -\delta)/t \rb
\rp$.  Thus Equation \eqref{Eqn:OffGamma} holds in the case when $M$
is non-compact as well.  Combining these estimates completes the
proof.  $\Box$

This theorem, in principle, gives the small-time asymptotics of the
heat kernel in great generality.  To get more concrete information,
one needs to be able to determine the small-time asymptotics of the
integral over $\GE$.  Fortunately, this is a well-studied type of
integral, called a Laplace integral, as we shall discuss shortly.

Finally, we have stopped with the first term of the Ben Arous
expansion only for convenience.  As much of that expansion can be kept
as desired, in which case the $c_0(x,z)c_0(z,y)+O(t)$ in the integrand
is replaced by a more general product of Taylor series.  However, it
is unclear how much additional information this really provides.  It
seems that relatively little is known about the functions $c_0$, and
higher-order coefficients in the Ben Arous expansion are even less
well-understood.  Further, including such terms means that we would also want
to determine higher-order terms in the asymptotic behavior of the
Laplace integral over $\GE$, which doesn't seem practical in general.
For these reasons, we content ourselves with the leading term.

\section{Understanding the Laplace integral} \label{s:laplace}

We wish to determine the asymptotics of the integral that appears in
Theorem \ref{THM:MainExpansion}.  To this end, we first review this
type of integral, from which we see that the behavior of $h_{x,y}$
near $\Gamma$ is the key factor.  Then we discuss the geometric
meaning of the behavior of $h_{x,y}$ in terms of the conjugacy of
minimal geodesics from $x$ to $y$.

\subsection{A brief discussion of Laplace asymptotics}
\label{s:laplace2}
We now discuss techniques for determining the small $t$
asymptotics of integrals of the type
\begin{equation}\label{Eqn:LaplaceIntegral}
\int_D f(x) e^{-g(x)/t} \, dx .
\end{equation}
Here $D$ is a compact set of $\bR^n$ having the origin in its
interior, $f$ is smooth in a neighborhood of $D$, and $g$ is a
function which is smooth in a neighborhood of $D$, is
zero at the origin, and is strictly positive on $D$ minus the origin.  (We also
assume the integral is with respect to Lebesgue measure; to treat any
other measure with a smooth density we can simply incorporate the
density into $f$.)  Our assumption that $g$ has zero as its minimum is
no loss of generality; for any $a\in \bR$ we have that
\[
\int_D f(x) e^{-(g+a)/t} \, dx = e^{-a/t} \int_D f(x) e^{-g(x)/t} \, dx .
\]

We start with the one-dimensional case.  Further, we assume that $g$
can be written as $x^{2m}$ for some integer $m\geq 1$.  Again, if this
can be accomplished by first performing a smooth change of
coordinates (and possibly shrinking $D$), we can just absorb the Jacobian into $f$.  While it is
not always possible to find such a change of coordinates, this is the most important case.  Then (see, for example,
\cite{KanwalEstrada})
\[
\int_D f(x) e^{-x^{2m}/t} \, dx = f(0)
\frac{\Gamma\lp1/(2m)\rp}{m}t^{1/(2m)} +O\lp t^{3/(2m)} \rp,
\quad\text{as $t\searrow 0$},
\]
(here ``$\Gamma$'' is the usual Gamma function, not the set of
midpoints of minimal geodesics).
We note that higher terms in this expansion are known, but in the
present context we continue to focus only on the leading term.

The higher dimensional situation is more complicated.  If we assume
that $g$ can be written as
\[
g(x) = \sum_{i=1}^n x_i^{2m_i},
\]
for some integers $1\leq m_1\leq m_2 \leq \cdots \leq m_n$, then the
expansion essentially decomposes as a product of one-dimensional
integrals.  This immediately gives
\begin{equation}\label{Eqn:DiagonalG}
\int_D f(x) e^{-g(x)/t} \, dx = t^{\frac{1}{2m_1}+\cdots+\frac{1}{2m_i}} \lb
f(0)\prod_{i=1}^n \frac{\Gamma\lp 1/2m_i\rp}{m_i} +O\lp t^{1/m_n} \rp
\rb .
\end{equation}
In particular, if the Hessian of $g$ is non-degenerate at the origin,
the Morse lemma guarantees that we can always find coordinates near
the origin in which $g$ is a sum of squares, and thus the above
expansion holds in these coordinates with $m_i=1$ for all $i$.
However, if the Hessian is degenerate, it will not necessarily be true
that $g$ can be put into the form of Equation~\eqref{Eqn:DiagonalG} by
a smooth change of coordinates.

Nonetheless, we recall that the \virg{splitting lemma} for smooth functions (which can be found in \cite{gromollmeyer}) allows us to split off non-degenerate directions and thus partially diagonalize $g$. In that spirit, the following result guarantees that, around an isolated degenerate critical point of corank 1, there always exists a coordinate set in which $g$ is diagonal. It is a generalization of the classical Morse lemma for nondegerate critical points and is a particular case of the splitting lemma just mentioned.
\bl \label{Lem:Splitting}
Let $g$ be a smooth 
 function on a neighborhood of the origin in $\R^n$, such that the origin is a local minimum of $g$ and the only critical point of $g$.  Assume that $g(0)=dg(0)=0$ and that $\dim \ker d^2g(0)=1$. Then there exists a diffeomorphism $\phi$ from a neighborhood of the origin to a neighborhood of the origin and a smooth function $\psi:\R\to \R$ such that
$$g(\phi(u))=\sum_{i=1}^{n-1} u_i^2+\psi(u_n), \qquad \text{where}\qquad \psi(u_n)=O(u_n^4).$$
\el

More generally, suppose that $g$ is equal to its Taylor series near
the origin.  Even this doesn't cover all possible cases (in
particular, if $g$ is smooth but not real-analytic), but it seems to
be the most general case for which there is a satisfactory theory.  In
the case where $g$ is equal to its Taylor series near the origin,
Arnold and his collaborators (see \cite{Arnold} and the references
therein) have given a powerful
analysis of the resulting asymptotics.  Briefly, if $g$ is
real-analytic with a unique minimum of zero at the origin then the
leading term in the expansion (assuming $f(0)\neq 0$) is of the form
$cf(0)t^{\alpha}|\log t|^m$ where $c$ is a positive constant, $\alpha$
is a positive rational, and $m$ is an integer between $0$ and $n-1$
inclusive.  Estimates on $\alpha$ and $m$ can be given in terms of
combinatorial information derived from which monomials in the Taylor series
of $g$ have non-zero coefficients (more precisely, one looks at
various features of the Newton diagram of $g$).  Moreover, generically
(in a sense which can be made precise) $\alpha$ and $m$ are determined
by this combinatorial information.

The above assumes that $g$ has an isolated minimum at the origin.
Suppose, instead, that $g$ assumes its minimum along some smooth
submanifold.  In this case, one can choose coordinates for the
minimum set and then extend them to coordinates near the minimum set
by adding coordinates for the normal bundle.  Then at each point of
the minimum set, one can try to apply the above analysis to the
corresponding fiber of the normal bundle, and then attempt to
integrate the result over the minimum set.  The simplest such case is
when $g$ is a Morse-Bott function, in which case the asymptotics on
each fiber will be just those corresponding to a non-degenerate
Hessian of the appropriate dimension (this is what we see, for
example, for the Heisenberg group in Section \ref{s:examples}),
although in general the situation can be more complicated.

In the case when this is not
possible (for example, if the minimum set has a more complicated
structure than a submanifold), a somewhat more general statement can
be made.  If $g$ is
real-analytic, one can use a resolution of singularities to reduce the
situation to that of a sum of integrals of the form given in Equation
\eqref{Eqn:LaplaceIntegral}, where in each term of the sum $g$ is a
monomial in the new coordinates and $f$ is a smooth function times the
absolute value of a monomial.  (Essentially, the resolution of
singularities amounts to a type of generalized change of coordinates
under which $g$ has this more restricted form.)  The small-time
asymptotics in such a case are again given by a rational power of $t$
times an integer power of $|\log t|$.  In contrast to the above case
of an isolated minimum, here there does not seem to be a way of
understanding the powers of $t$ and $|\log t|$ without determining the
resolution of singularities and computing the asymptotics of each of
the resulting integrals.

The interested reader is referred to the references above for complete
details, or to Sections 3.5 and 3.6 of \cite{Neel} which contain a
more detailed summary of these results (and which seems too much of a
digression to repeat here).

\subsection{Conjugacy and the behavior of $h_{x,y}$}

We now discuss how the behavior of $h_{x,y}$ near its minima relates
to the structure of the minimal geodesics from $x$ to $y$,
specifically, to the conjugacy of these geodesics.  Suppose we have
distinct points $x$ and $y$ such that every minimizer from $x$ to $y$
is strongly normal. We begin by introducing notation.

Consider any point $z_0\in \Gamma$, which corresponds to some minimal
geodesic $\gamma$ from $x$ to $y$.  Then there is a unique covector
$\lambda\in T^*_{x}M$ such that 
$\Exp_x(2\lambda)=y$ and that $\Exp_x(2\lambda,t)$ for
$t\in[0,1]$ parametrizes $\gamma$.   (Recall that
$\Exp_x(\lambda)=\pi\circ e^{\vec{H}}(x,\lambda)$.)

Let $\lambda(s)$ be a smooth curve of covectors $\lambda:(-\eps, \eps)
\rightarrow T^{*}_x M$ (for some small $\eps>0$) such that
$\lambda(0)=\lambda$ and the derivative never vanishes.  Thus
$\lambda(s)$ is a one-parameter family of perturbations of $\lambda$
which realizes the first-order perturbation $\lambda^{\prime}(0) \in
T_{\lambda}\lp T^{*}_x M\rp$.    Also, we let $z(s)= \Exp_x
(\lambda(s))$, so that $z(0)=z_0$.  Because $\Exp_x$ is a
diffeomorphism from a neighborhood of $\lambda$ to a neighborhood of
$z_0$, we see that the derivative of $z(s)$ also never vanishes.  Thus
$z(s)$ is a curve which realizes the vector $z^{\prime}(0)\in
T_{z_0}M$.  Further, we've established an isomorphism of the vector
spaces $T_{\lambda}\lp T^{*}_x  M\rp$ and $T_{z_0}M$ by mapping
$\lambda^{\prime}(0)$ to $z^{\prime}(0)$, except that we've excluded
the origin by insisting that both vectors are non-zero.

We say that $\gamma$ is conjugate in the direction
$\lambda^{\prime}(0)$ (or with respect to the perturbation
$\lambda^{\prime}(0)$) if $\frac{d}{ds}\Exp(2\lambda(s))|_{s=0} =0$.
Note that this only depends on $\lambda^{\prime}(0)$.  We say that the
Hessian of $h_{x,y}$ at $z_0$ is degenerate in the direction
$z^{\prime}(0)$ if $\frac{d^2}{ds^2}h_{x,y}(z(s))|_{s=0} =0$.  This
last equality is equivalent to writing the Hessian of $h_{x,y}$ as a
matrix in some smooth local coordinates, applying it as a quadratic
form to $z^{\prime}(0)$ expressed in these coordinates, and getting
zero.  This equivalence, as well the fact that whether the result is
zero or not depends only on $z^{\prime}(0)$, follows from the fact
that $z_0$ is a critical point of $h_{x,y}$.

The point of the the next theorem is that conjugacy in the direction
$\lambda^{\prime}(0)$ is equivalent to degeneracy in the direction
$z^{\prime}(0)$.  Thus the Hessian of $h_{x,y}$ encodes information
about the conjugacy of $\gamma$, and it is a more geometric object
than it might seem at first.

\begin{THM}\label{THM:HAndConj}
Let $M$ be a  complete sub-Riemannian
manifold in the sense of Section \ref{s:srg} or Appendix \ref{appendix}, and let $x$ and $y$ be distinct
points such that every minimal geodesic from $x$ to $y$ is strongly
normal.  Define $\Gamma$, $z_0\in \Gamma$, $h_{x,y}$ and the curves $\lam(s), z(s)$ as above.
Then  
\bi
\iii[(i)]  $\gamma$ is conjugate
if and only if the Hessian of
$h_{x,y}$ at $z_0$ is degenerate. 
\iii[(ii)] In particular $\gamma$ is
conjugate in the direction $\lambda^{\prime}(0)$ if and only if the
Hessian of $h_{x,y}$ at $z_0$ is degenerate in the corresponding
direction $z^{\prime}(0)$. 
\iii[(iii)] The dimension of the space of
perturbations for which $\gamma$ is conjugate is equal to the dimension of the kernel of the
Hessian of $h_{x,y}$ at $z_0$. 
\ei
\end{THM}

\begin{proof}
We know that there is a unique shortest geodesic from $y$ to $z(s)$
for all $s\in(-\eps,\eps)$, assuming $\eps$ small enough.  Let
$\eta(s)$ be the corresponding smooth curve of covectors in $T^{*}_y M$
(that is, $\Exp_y(t \eta(s))$ for $t\in[0,1]$ parametrizes the
minimal geodesic from $y$ to $z(s)$).  Let $\tilde{\lambda}(s)$ and $\tilde{\eta}(s)$ be
the images of $\lambda$ and $\eta$, respectively, under the corresponding
Hamiltonian flow on the cotangent bundle.  
We see that $\tilde{\lambda}(0)+ \tilde{\eta}(0)=0$.

Observe that $d(E(x,\cdot))|_{z(s)} =
\tilde{\lambda}(s)$ and $d(E(y,\cdot))|_{z(s)}
= \tilde{\eta}(s)$. (Here $d$ stands for the differential.) It follows that
\[
dh_{x,y}|_{z(s)} = \tilde{\lambda}(s) +\tilde{\eta}(s) \quad \lp \in
T^{*}_{z(s)}M\rp .
\]
Next note that $\gamma$ is conjugate in the direction
$\lambda^{\prime}(0)$ if and only if $\tilde{\lambda}(s)
+\tilde{\eta}(s) = O(s^2)$, as follows directly from consideration of
the exponential map.  Thus $\gamma$ is conjugate in the direction
$\lambda^{\prime}(0)$ if and only if $dh_{x,y}|_{z(s)}=O(s^2)$.

We claim that $dh_{x,y}|_{z(s)}$ is $O(s^2)$ if and only if
$h_{x,y}(z(s))$ is $O(s^3)$.  Equivalently, the derivative of
$dh_{x,y}$ (as a one-form) in the $z^{\prime}(0)$ direction is zero if
and only if its pairing with $z^\prime(0)$ is zero.  This relationship
is most easily expressed in local coordinates.  Let $H$ be the
$n\times n$ matrix for the Hessian of $h_{x,y}$ at $z_0$ in some local
coordinates, and let $v$ be $z^{\prime}(0)$ expressed in these
coordinates.  Then the derivative of $dh_{x,y}$ in the $z^{\prime}(0)$
direction is $Hv$, which we think of an operator on vectors $u$ by
$\ip{Hv}{u}$ where is the standard Euclidean inner product for these
coordinates, or equivalently by $u^THv$ where $u$ and $v$ are written
as column-vectors.

The claim now follows from the following simple fact from linear
algebra: for any symmetric and positive semi-definite $n\times n$ real matrix $A$, we have that, for any $x\in\bR^n$,
$\ip{Ax}{x}=0$ if and only if $Ax=0 \in\bR^n$, where
$\ip{\cdot}{\cdot}$ is the standard Euclidean inner product. 
Because $A$ is symmetric, we can find an
orthonormal basis 
$v_i$, $i=1,\ldots, n$ for $\R^n$ cosnisting of
 eigenvectors of $A$ with
corresponding eigenvalues $\lambda_i\geq 0$.  
Then, writing $x=\sum_{i=1}^n x_i v_i$, the above fact follows from the identities
$Ax=\sum_{i=1}^n \lam_i x_i$ and $\la Ax,x \ra=\sum_{i=1}^n \lam_i x_i^2$.

Since the Hessian of $h_{x,y}$ at $z_0$
is symmetric and positive semi-definite, the claim follows.  Thus we
have proven statement $(ii)$ in the theorem (and $(i)$ a fortiori), namely that $\gamma$
is conjugate in the direction $\lambda^{\prime}(0)$ if and only if the
Hessian of $h_{x,y}$ at $z_0$ is degenerate in the direction
$z^{\prime}(0)$.

Statement $(iii)$ is an immediate consequence of $(ii)$ plus the fact that the correspondence between $\lambda^{\prime}(0)$ and $z^{\prime}(0)$ gives an isomorphism of
vector spaces between $T_{\lambda}\lp T^{*}_x  M\rp$ and $T_{z_0}M$,
as discussed just before the theorem.
\end{proof}

We now briefly discuss the situation of higher-order derivatives of
the exponential map and higher-order derivatives if $h_{x,y}$.  This
situation is more complicated than what we just saw for lower-order
derivatives.

Recall that $\Exp_x (2\lambda) = y$.  Further, consider
$$
\frac{d^m}{ds^m} \Exp_x (2\lambda(s))|_{s=0} .
$$
For $m=1$, this is zero if and only if $y$ is conjugate to $x$ along
the geodesic through $z_0$ in the direction of $\lambda^{\prime}(0)$.
If this first derivative is zero, then the number of higher order
derivatives which vanish describes, in a sense, how conjugate $y$ is
to $x$ with respect to the perturbation $2\lambda(s)$.  (Of course,
it's possible for all derivatives to vanish; for example, if
$2\lambda(s)$ describes a one-parameter family of minimal geodesics
from $x$ to $y$, as occurs for the Heisenberg group.)  We can compare
the vanishing of these derivatives to the vanishing of the derivatives
$\frac{d^k}{ds^k}h_{x,y}(z(s))$.

Suppose that, for some positive integer $m$, we have that $\Exp_x
(2\lambda(s))=w s^m +O(s^{m+1})$ for some non-zero $w\in\bR^n$ in some
(smooth) system of coordinates around $y$ (so that $y$ is at the
origin of these coordinates).  If this holds in one such system, in
holds in any other such system with $w$ re-expressed in the new
coordinates.  Because the exponential map is a diffeomorphism from a
neighborhood of each $\tilde{\eta}(s)\in T^{*}_{z(s)}M$ to a
neighborhood of $y$, we see that this expansion for  $\Exp_x
(2\lambda(s))$ is equivalent to having
\[
\tilde{\lambda}(s)+\tilde{\eta}(s) = dh_{x,y}|_{z(s)} = vs^m +
O(s^{m+1})\in T^{*}_{z(s)}M
\]
 for some non-zero one-form $v$ written with respect to some (smooth)
system of coordinates around $z_0$.  Again, if this holds for one such
system, it holds for any other such system with $v$ re-expressed
relative to the new coordinates.

Thus, the one-form $dh_{x,y}|_{z(s)}$ vanishes to the same order as
the derivatives of $\Exp_x (2\lambda(s))$.  However, when we look at
$h_{x,y}(z(s))$, we see that
\begin{align*}
h_{x,y}&(z(s)) - h_{x,y}(z_0) = \int_0^s dh_{x,y}|_{z(t)} (z^{\prime}(t)) \, dt=\\
 &= \int_0^s \lp v(z^{\prime}(t)) t^m+O(t^{m+1}) \rp  \, dt 
  = \frac{1}{m+1} v(z^{\prime}(0)) s^{m+1}+O(s^{m+2}) .
\end{align*}
So we have that $h_{x,y}(z(s)) - h_{x,y}(z_0) = c s^{m+1}+O(s^{m+2})$
for non-zero $c$ if and only if $v(z^{\prime}(0))\neq 0$.  For $m=1$,
it is always the case that $v(z^{\prime}(0))\neq 0$, as we saw in the
previous theorem.  However, for $m>1$, this need not be true.  In such
a case we can only conclude that $h_{x,y}(z(s)) - h_{x,y}(z_0) =
O(s^{m+2})$, and the exact order of vanishing of the derivatives of
$h_{x,y}$ is unclear in general.

\brem In the special case when $M$ is two-dimensional and $z_0$ is an isolated minimum such that $h_{x,y}$ vanishes to finite order at $z_0$, the relationship is simpler. Namely, because the Hessian of $h_{x,y}$ is clearly non-degenerate along the direction of $\gamma$, we can apply Lemma \ref{Lem:Splitting} to write $h_{x,y} = u_1^2 +g(u_2)$ for some coordinates $u_i$ around $z_0$ and some smooth function $g$.  Then if $z(s)$ corresponds to the curve $(u_1,u_2)=(0,s)$, we see by direct computation that $v(z^{\prime}(0))\neq 0$. In this way, the degree of degeneracy of the Hessian and the degree of conjugacy correspond precisely in this case.
\erem

We also note that, at the opposite extreme, there is again a nice
correspondence between the behavior of the exponential map and of
$h_{x,y}$.  Namely, $\Exp_x (2\lambda(s))=y$ for all
$s\in(-\eps,\eps)$ if and only if $h_{x,y}(z(s)) = h_{x,y}(z_0)$ for
all $s\in(-\eps,\eps)$, as follows directly from Lemma
\ref{Lem:AboutGamma}.

All of this seems to indicate that, loosely speaking, the more
conjugate the geodesic through $z_0$ is, the more degenerate $h_{x,y}$
is at $z_0$.  However, it also seems that looking at curves through
$z_0$ corresponding to one-parameter perturbations of the geodesic is
too naive in general, and that a more sophisticated approach is needed to
describe the exact relationship between the higher order derivatives
of the exponential map and higher order terms in the Taylor series of
$h_{x,y}$.  As we do not need anything beyond the results of Theorem
\ref{THM:HAndConj} in what follows (except perhaps to give geometric
intuition to $h_{x,y}$), we do not pursue this direction any further.
(In light of the above, it seems that the claims about higher-order
derivatives in Lemma 3.1 of \cite{Neel} are over-simplified.
Fortunately, in that paper, as in the present, only the content of
Theorem \ref{THM:HAndConj} is used in subsequent arguments.  The
higher-order relationship serves only to provide a more geometric
meaning to the behavior of $h_{x,y}$.)

\section{General consequences of Laplace asymptotics}
\label{s:conslapl}

We are now in a position to see what the theory of Laplace asymptotics
summarized in the previous section gives when applied to the integral
in Theorem \ref{THM:MainExpansion}.  The ideas expand upon those of
Section 5.3 of \cite{Hsu}, where inequality
\eqref{Eqn:DoubleSidedBounds} of Theorem \ref{THM:ImprovedExpansion}
is given in the Riemannian case.  We note that the results we give in
this section, most interestingly those which depend on whether or not
$x$ and $y$ are conjugate along a given geodesic, are also valid in
the Riemannian case.

We begin with a basic lemma.  For this lemma, we say that the $u_i$
are ``coordinates around $z_0$'' if they are coordinates on some
neighborhood of $z_0$ such that $u_i(z_0)=0$ for all $i$.
Inequalities for such coordinates are understood to hold on some such
neighborhood.

\begin{lemma}
Under the assumptions of Theorem \ref{THM:HAndConj}, let $z_0$ be
any point of $\Gamma$.  Then there exist smooth coordinates
$u_1,\ldots, u_n$ around $z_0$ such that
\bqn\label{eq:E} 
h_{x,y}(u_1,\ldots,u_n) \geq \frac{1}{4}d^{2}(x,y) + u_1^2.
\eqn
Also, there exist smooth coordinates
$v_1,\ldots, v_n$ around $z_0$ such that
\[
h_{x,y}(v_1,\ldots,v_n) \leq \frac{1}{4}d^{2}(x,y)+ v_1^2 +\cdots + v_n^2 .
\]
Finally, if the geodesic from $x$ to $y$ passing through $z_0$ is
conjugate, then the $v_i$ can be chosen so that
\[
h_{x,y}(v_1,\ldots,v_n) \leq \frac{1}{4}d^{2}(x,y)+ v_1^2 +\cdots + v_{n-1}^2 +v_n^4.
\]
\end{lemma}
\begin{proof}

For $z$ in some neighborhood of $z_0$, let $u_1(z)=u_1=
d(x,z)-(d(x,y)/2)$.  If the neighborhood is small enough, this is a
smooth function with non-vanishing derivative (since $d(x,z)$ has
these properties as a function of $z$) and $u_1(z_0)=0$.  Thus it is a
valid coordinate, and we can complete this to a full set of
coordinates around $z_0$.  Further, the triangle inequality gives
\[
d(y,z) \geq d(x,y)-d(x,z) = \frac{1}{2}d(x,y)-u_1 .
\]
Thus we compute
\begin{align*}
h_{x,y}&(z) = \frac{1}{2}\lb d(x,z)^2 + d(y,z)^2\rb\\ &
\geq \frac{1}{2}\lb \lp u_1+\frac{1}{2}d(x,y)\rp^2 +\lp
\frac{1}{2}d(x,y)-u_1\rp^2  \rb
= u_1^2 +\frac{1}{4}d(x,y)^2 .
\end{align*}
which gives the estimate \eqref{eq:E}.

For the second estimate, recall that $h_{x,y}$ is smooth and assumes
its minimum at $z_0$, and thus the derivative of $h_{x,y}$ vanishes at
$z_0$.  It follows that for any system of coordinates $w_i$ around
$z_0$, there is a small enough neighborhood of $z_0$ and positive
constant $C$ such that
\[
h_{x,y}(v_1,\ldots,v_n) \leq \frac{1}{2}E(x,y)+ C\lp w_1^2 +\cdots + w_n^2\rp
\]
on this neighborhood.  Thus, we can simply rescale the $w_i$ to get
coordinates $v_i$ as required.

The proof of the final inequality is based on Lemma \ref{Lem:Splitting}.

Because the geodesic through $z_0$ is conjugate, the Hessian of
$h_{x,y}$ at $z_0$ cannot have full rank, as we see from Theorem \ref{THM:HAndConj}.  First assume that the rank is exactly $n-1$.  Then Lemma \ref{Lem:Splitting}
shows that there are
coordinates $v_i$ around $z_0$ such that
\[
h_{x,y}(v_1,\ldots,v_n) = \frac{1}{2}E(x,y)+ v_1^2 +\cdots + v_{n-1}^2 +O(v^4_n)
\]
Then, after possibly rescaling $v_n$, we see that the desired inequality holds.

Next assume the Hessian of $h_{x,y}$ has rank less than $n-1$.  Then
let $\phi$ be a smooth function on a neighborhood of $z_0$ which is
non-negative, zero at $z_0$ (hence with vanishing derivative at
$z_0$), and such that $h_{x,y}+\phi$ has Hessian of rank $n-1$ at $z_0$
($\phi$ looks like a sum of squares of an appropriate number of
coordinates, for example).  Applying the previous result to
$h_{x,y}+\phi$ shows that there are coordinates $v_i$ around $z_0$ such
that
\[
\lp h_{x,y}+\phi\rp (v_1,\ldots,v_n) \leq \frac{1}{2}E(x,y)+ v_1^2
+\cdots + v_{n-1}^2 +v_n^4 .
\]
Since $\phi$ is non-negative, the desired estimate for $h_{x,y}$ follows.
\end{proof}

These estimates allow us to say more about the integral appearing in
Theorem \ref{THM:MainExpansion}.

\begin{THM}\label{THM:ImprovedExpansion}\label{t:mainmain}
With the same assumptions and notation as Theorem
\ref{THM:MainExpansion}, we have that for any sufficiently small neighborhood $N(\Gamma)$
\bqn \label{eq:neww}
  \qquad p_t(x,y) = \int_{\GE} \lp\frac{2}{t}\rp^{n} e^{-h_{x,y}(z)/t}\lp
c_0(x,z)c_0(z,y)+O(t)\rp \, \mu(dz) .
\eqn
Again, the ``$O(t)$'' term in the integral is uniform over $\GE$.
Also, there exist positive constants $C_i$, and $t_0$ (depending on
$M$, $x$, and $y$) such that
\begin{equation}\label{Eqn:DoubleSidedBounds}
\frac{C_1}{t^{n/2}}  e^{-d^2(x,y)/4t} \leq p_t(x,y)
\leq\frac{C_2}{t^{n-(1/2)}}  e^{-d^2(x,y)/4t}
\end{equation}
for $0<t<t_0$.
Further, if $x$ and $y$ are conjugate along any minimal geodesic
connecting them, then (perhaps after changing $t_0$), we have
\bqn\label{eq:stima2}
p_t(x,y)\geq \frac{C_3}{t^{(n/2)+(1/4)}}  e^{-d^2(x,y)/4t}
\eqn
for $0<t<t_0$.  Finally, if $x$ and $y$ are not conjugate along any
minimal geodesic joining them, then
\bqn \label{eq:stima3}
p_t(x,y) = \frac{C_4+O(t)}{t^{n/2}}  e^{-d^2(x,y)/4t} .
\eqn
\end{THM}
\brem
Notice that Corollary \ref{t:main} in the Introduction is a direct
consequence of formulas \eqref{eq:neww} and \eqref{Eqn:DiagonalG}.
Corollary \ref{corollarioMAIN} contains the estimates
\eqref{Eqn:DoubleSidedBounds}, \eqref{eq:stima2} and
\eqref{eq:stima3}.
\erem
\begin{proof}  We begin with the general bounds on $p_t(x,y)$.  Choose
any $z_0\in \Gamma$ and let $V\subset \GE$ be some neighborhood on
which there are coordinates $v=(v_1,\ldots,v_n)$ as in the previous
lemma (that is, $h_{x,y}$ is estimated by the sum of squares of the
$v_i$).  Since the integrand in Theorem \ref{THM:MainExpansion} is
positive for sufficiently small $t$, we have that
\begin{multline*}
\int_{\GE} \lp\frac{2}{t}\rp^{n} e^{-h_{x,y}(z)/t}\lp
c_0(x,z)c_0(z,y)+O(t)\rp \, \mu(dz)
\\ \geq
\lp\frac{2}{t}\rp^{n} e^{-E(x,y)/2t}  \int_{V} e^{-\lp
v_1^2+\cdots+v_n^2\rp/t}\lp
c_0(x,v)c_0(v,y)+O(t)\rp \, \mu(dv)
\end{multline*}
for all sufficiently small, positive $t$.

Because $\mu$ is a smooth volume and $v$ a smooth coordinate system,
we know that there is a smooth, positive function $F$ such that
$\mu(dv)= F(v)dv_1\cdots dv_n$.  Then the results of the previous
section, namely Equation \eqref{Eqn:DiagonalG}, show that
\begin{multline*}
\int_{V} e^{-\lp v_1^2+\cdots+v_n^2\rp/t}\lp
c_0(x,v)c_0(v,y)+O(t)\rp  F(v) \, dv_1\cdots dv_n \\
= t^{n/2}\lb F(0)\lp c_0(x,z_0)c_0(z_0,y)+O(t)\rp\pi^{n/2}+O(t)\rb
\end{multline*}
(where we've used that $\Gamma(1/2)=\sqrt{\pi}$).  Note that there's
no difficulty handling the $O(t)$ in the integrand since we simply
estimate it by $|O(t)|\leq Ct$ for some positive $C$ and factor the
$t$ out of the integral.  Putting this together with the fact that
$F(0)c_0(x,z_0)c_0(z_0,y)$ is positive, we see that there exist
positive $C_1$ and $t_0$ such that
\[
\int_{\GE} \lp\frac{2}{t}\rp^{n} e^{-h_{x,y}(z)/t}\lp
c_0(x,z)c_0(z,y)+O(t)\rp \, \mu(dz) \geq \frac{C_1}{t^{n/2}}  e^{-E(x,y)/2t}
\]
for $0<t<t_0$.  Comparing this to Theorem \ref{THM:MainExpansion}, we
note that the $o\lp \exp\lb \frac{-E(x,y)/2 -\delta}{t} \rb \rp$ term
is dominated by the right-hand side of the above inequality.  Thus,
after possibly adjusting $C_1$ and $t_0$, we see that the relevant
inequality in the theorem holds.

For the other side of the first inequality, note that we can find
coordinates $u_i$ as in the previous lemma around every point of
$\Gamma$, and each of these systems of coordinates is defined on some
open neighborhood.  Because $\Gamma$ is compact, there is a finite set
of such neighborhoods which cover $\Gamma$; denote them by
$U_1,\ldots, U_m$ and the corresponding systems of coordinates by
$u_j=(u_{j,1},\ldots,u_{j,n})$ for $j=1,\ldots,m$.  Now choose $\GE$
small enough so that $\GE\subset \cup_{j=1}^m U_j$.  Then we have
\begin{multline*}
\int_{\GE} \lp\frac{2}{t}\rp^{n} e^{-h_{x,y}(z)/t}\lp
c_0(x,z)c_0(z,y)+O(t)\rp \, \mu(dz)
\\ \leq \sum_{j=1}^m
\lp\frac{2}{t}\rp^{n} e^{-E(x,y)/2t}  \int_{U_j} e^{-u_{j,1}^2/t}\lp
c_0(x,u_j)c_0(u_j,y)+O(t)\rp \, \mu(du_j)
\end{multline*}
for all sufficiently small, positive $t$.  As above, $\mu$ is a smooth
volume, and Equation \eqref{Eqn:DiagonalG} gives that, for each $j$,
there is a positive constant $K_j$ such that
\[
\int_{U_j} e^{-u_{j,1}^2/t}\lp c_0(x,u_j)c_0(u_j,y)+O(t)\rp \,
\mu(du_j) = \sqrt{t}\lp K_j +O(t)\rp.
\]
Summing $j$ from $1$ to $m$ allows us to conclude that there exists
positive $C_2$ such that, after possibly making $t_0$ smaller,
\[
\int_{\GE}\hspace{-0.2cm} \lp\frac{2}{t}\rp^{n} \hspace{-0.2cm}e^{-h_{x,y}(z)/t}\lp
c_0(x,z)c_0(z,y)+O(t)\rp \, \mu(dz) \leq \frac{C_2}{t^{n-(1/2)}}  e^{-E(x,y)/2t}
\]
for $0<t<t_0$.  Again, comparing this to Theorem
\ref{THM:MainExpansion}, we see that the other side of the first
inequality in the theorem holds, after possibly adjusting $C_2$ and
$t_0$.

The two-sided inequality we've just proved now shows that the term $o\lp
\exp\lb \frac{-E(x,y)/2 -\delta}{t} \rb \rp$ in Theorem
\ref{THM:MainExpansion} is unnecessary; it can be ``included'' in the
$O(t)$ term in the integral (as we've already taken advantage of
above).  This establishes the first claim in the theorem.

Now we consider the case when $x$ and $y$ are conjugate along some
minimal geodesic.  Suppose that $z_0$ is the midpoint of this
geodesic.  Then we can find coordinates $v_i$ around $z_0$, defined on
some neighborhood $V\subset \GE$, such that
\[
h_{x,y}(v_1,\ldots,v_n) \leq \frac{1}{2}E(x,y)+ v_1^2 +\cdots +
v_{n-1}^2 +v_n^4 .
\]
Analogous to the previous lower bound, we have that
\begin{multline*}
\int_{\GE} \lp\frac{2}{t}\rp^{n} e^{-h_{x,y}(z)/t}\lp
c_0(x,z)c_0(z,y)+O(t)\rp \, \mu(dz)
\\ \geq
\lp\frac{2}{t}\rp^{n} e^{-\frac{E(x,y)}{2t}}  \int_{V} e^{-\lp
v_1^2+\cdots+v_{n-1}^2+v_n^4\rp/t}\lp
c_0(x,v)c_0(v,y)+O(t)\rp \, \mu(dv),
\end{multline*}
for all sufficiently small, positive $t$.  Equation
\eqref{Eqn:DiagonalG} (along with smoothness of $\mu$ and positivity
of the $c_0$) then shows that, for some positive constants $C_3$ and
$t_0$ (possibly different from before),
\begin{align*}
\int_{V} e^{-\lp v_1^2+\cdots+v_{n-1}^2+v_n^4\rp/t}\lp
c_0(x,v)c_0(v,y)+O(t)\rp \, \mu(dv)\\ \geq t^{(n/2)-(1/4)} \lp
C_3+O\lp\sqrt{t}\rp\rp,
\end{align*}
for $0<t<t_0$.  Combining these estimates and the first claim in the
theorem, we see that, after possibly adjusting $C_3$ and $t_0$,
\[
p_t(x,y)\geq \frac{C_3}{t^{(n/2)+(1/4)}}  e^{-E(x,y)/2t},
\]
for $0<t<t_0$.

Finally, we suppose that $x$ and $y$ are not conjugate along any
minimal geodesic joining them.  Then for any $z_0\in\Gamma$, Theorem
\ref{THM:HAndConj} and the Morse lemma imply that $z_0$ is isolated.
Since
$\Gamma$ is compact, we see that in fact $\Gamma$ consists of finitely
many points, say $z_1,\ldots, z_m$ (so there are only finitely many
minimal geodesics from $x$ to $y$).  Further, we can find coordinates
$u_{j,1},\ldots, u_{j,n}$ around each $z_j$, on some neighborhood
$U_j$, such that
\[
h_{x,y}(u_{j,1},\ldots,u_{j,n}) = \frac{1}{2}E(x,y)+ u_{j,1}^2 +\cdots
+ u_{j,n}^2 \quad\text{on $U_j$,}
\]
and $\GE$ is the disjoint union of the $U_j$ (for small enough
$U_j$).  Thus, using the first claim in the theorem,
\begin{align}\nn
p_t(x,y) = \lp\frac{2}{t}\rp^{n} & e^{-E(x,y)/2t}
 \sum_{j=1}^m  \int_{U_j} e^{-(u_{j,1}^2+\cdots+u_{j,n}^2)/t}  \times \\
&\times c_0(x,u_j)c_0(u_j,y)+O(t)  \, \mu(du_j) .
\end{align}
We have that $\mu(du_j)=F_j(u_j)du_{j,1}\cdots du_{j,n}$ for smooth,
positive $F_j$.  As above, we compute
\begin{multline*}
\int_{U_j} e^{-\lp u_{j,1}^2+\cdots+u_{j,n}^2\rp/t}\lp
c_0(x,u_j)c_0(u_j,y)+O(t)\rp  F_j(u_j) \, du_{j,1}\cdots du_{j,n} \\
= t^{n/2}\lb F_j(0)\lp c_0(x,z_j)c_0(z_j,y)+O(t)\rp\pi^{n/2}+O(t)\rb .
\end{multline*}
Summing over $j$, we have
\[
p_t(x,y) = \frac{C_4+O(t)}{t^{n/2}}  e^{-E(x,y)/2t},
\]
where $C_4 = \lp 4\pi\rp^{n/2}\sum_{j=1}^m F_j(0)
c_0(x,z_j)c_0(z_j,y)$, which is clearly positive.
\end{proof}

One consequence of this result is that the exponent of $1/t$ in the
small-time expansion of $p_t(x,y)$ ``sees'' whether or not $x$ and $y$
are conjugate along any minimal geodesic.  Said differently, the
exponent of $t$ detects the part of the cut locus of $x$ which comes
from conjugacy (assuming that the necessary geodesics are strictly
normal, of course).  That naturally leads to the question of what
happens at cut points which are not conjugate.

We first note that, if $y$ is not in the cut locus of $x$, then the
results of this analysis fit nicely with the expansion of Ben Arous,
which applies in a neighborhood of $y$.  In this case, there is a
single minimal geodesic from $x$ to $y$ and it is not conjugate.  Let
$z_1$ be the midpoint.  Then the same analysis as in the last part of
the previous proof (just with $m=1$) shows that
\[
p_t(x,y) =   \lp F(z_1) c_0(x,z_1)c_0(z_1,y)+O(t) \rp \frac{\lp
4\pi\rp^{n/2}}{t^{n/2}}  e^{-E(x,y)/2t} ,
\]
where $F(z_1)$ is the density of $\mu$ with respect to coordinates
which make the Hessian of $h_{x,y}$ at $z_1$ the identity matrix.
Since the Ben Arous expansion applies to $p_t(x,y)$, we also have
\[
p_t(x,y) =\lp c_0(x,y)+O(t) \rp \frac{1}{t^{n/2}}  e^{-E(x,y)/2t} .
\]
So in this case, Theorem \ref{THM:ImprovedExpansion} provides a
relationship between $c_0(x,y)$ on the one hand, and $c_0(x,z_1)$,
$c_0(z_1,y)$, and second-order behavior of $h_{x,y}$ at $z_1$ (which
is encoded by $F(z_1)$) on the other.

Now suppose that $y$ is in the cut locus of $x$, but that none of the
minimal geodesics from $x$ to $y$ are conjugate (and the assumptions
of Theorem \ref{THM:ImprovedExpansion} hold, of course).  Let
$\gamma_1(s)$ be one such geodesic, parametrized by arc-length so that
$\gamma_1(0)=x$ and $\gamma_1(d(x,y))=y$.  Then we claim that
$\lim_{s\nearrow d(x,y)}c_0(x,\gamma_1(s))$ exists and is positive,
and we denote it $\alpha_1$.  This follows from the relationship
between $c_0(x,\gamma_1(s))$ and $c_0(x,\gamma_1(s/2))$,
$c_0(\gamma_1(s/2),y)$, and $F(\gamma_1(s/2)$ just discussed, and that
fact that these last three quantities are continuous in $s$ and remain
positive.  (Indeed, we've already seen in the proof of Theorem
\ref{THM:ImprovedExpansion} that $\alpha_1=\lp 4\pi\rp^{n/2}F(z_1)
c_0(x,z_1)$ $c_0(z_1,y)$ where $z_1=\gamma_1(d(x,y)/2)$.)  Alternatively,
one can think of lifting a neighborhood of $\gamma_1([0,s])$ to a
``local'' universal cover and then applying the Ben Arous expansion.

Continuing, we let $\gamma_2,\ldots,\gamma_m$ be the other minimal
geodesics from $x$ to $y$, where we know that there can only be
finitely many and that $m$ must be at least $2$.  We let $\alpha_2,
\ldots,\alpha_m$ be the associated limits of $c_0(x,\cdot)$ along
these geodesics, analogous to $\alpha_1$.  Then the final part of the
proof of Theorem \ref{THM:ImprovedExpansion} shows that
\[
p_t(x,y) = \lb \sum_{j=1}^m \alpha_j +O(t)\rb \frac{1}{t^{n/2}}
e^{-E(x,y)/2t} .
\]

The point of relating the coefficient of $t^{-n/2}  e^{-d^{2}(x,y)/4t}$
in the above to the $c_0$ along the $\gamma_j$ is that we see that
this coefficient is discontinuous at $y$.  That is, for any
$\gamma_j$, we know that $c_0(x,\gamma_j(s))$ is continuous in a
neighborhood of $\gamma_j(s)$ as long as $0<s<d(x,y)$.  However, when
$s$ increases to $d(x,y)$, the value of this coefficient ``jumps up''
to the sum of the $\alpha_j$.  Thus, points which are not in the cut
locus of $x$ and points that are but are not conjugate to $x$ along
any minimal geodesics both have small-time heat kernel expansions that
look like a constant times $t^{-n/2}  e^{-d^{2}(x,y)/4t}$.  These two
types of points can be distinguished by whether or not the coefficient
(the constant) is continuous at the point in question.  However, if
one has that much information about the small-time heat kernel
asymptotics in a neighborhood of a point $y$, then presumably one
already understands $d(x,\cdot)$ near $y$, from which one should be
able to understand the local structure of the cut locus.  Thus looking
at this coefficient, from the perspective of locating the cut locus,
seems unlikely to be of much help.

This potentially stands in contrast to the case when $y$ is conjugate
to $x$ along a minimal geodesic, in which case only the power of $t$
appearing in the expansion at the point $y$ needs to be determined (in
order to conclude that $y$ is conjugate to $x$ along a minimal
geodesic).

\section{Examples} \label{s:examples}
In this section we discuss our results in some examples of 2-step sub-Riemannian structures. In these cases, an integral expression of the heat kernel (which can be explicitly written in some cases) has been found in \cite{bealsgaveaugreiner}. 

In the first example, namely the Heisenberg group, we briefly compute the Hessian of the hinged energy function $h_{x,y}$ when $x$ is the origin and $y$ is a point on the cut locus. In this case, being that both  the optimal synthesis and the heat kernel known explicitly, we verify the results of Theorem \ref{t:mainmain}.

The second example is the free nilpotent sub-Riemannian structure with growth vector (3,6). 
Here we use a ``reverse'' argument, starting from the formula for the heat kernel to find the asymptotics for points belonging to the vertical subspace, where all points are both cut and conjugate. This asymptotic agrees with the fact there exists a one parameter family of optimal geodesics that reach this point (for a detailed discussion about the optimal synthesis see \cite{nilpotent36}).

In this section the heat kernel is meant  for the intrinsic sub-Laplacian, i.e.\ it is computed with respect to the Popp volume. For the cases treated in this section this volume is proportional to the left Haar measure and is proportional to the Lebesgue measure in the standard system of coordinates we are using.

\subsection{Formula for the heat kernel in the 2-step case} \label{s:gaveau}
In this section we recall the expression of the heat kernel of the intrinsic sub-Laplacian associated with a  2-step nilpotent structure, that has been found in \cite{bealsgaveaugreiner}. 
Then we rewrite it to have a convenient expression on the \virg{vertical subspace}.

Consider on $\R^n$ a 2-step nilpotent structure of rank $k<n$, where $X_{1},\ldots,X_{k}$ is an orthonormal frame. Once a smooth complement $\mc{V}$ for the distribution is chosen (i.e.\ $T_{q}\R^n=\distr_{q}\oplus \mc{V}_{q}$, for all $q\in \R^n$) we can complete an orthonormal frame to a global one $X_{1},\ldots,X_{k},Y_{1},\ldots,Y_{m}$, where $m=n-k$ and $\mc{V}_{q}=\tx{span}_{q}\{Y_{1},\ldots,Y_{m}\}$. Since the structure is nilpotent, we can assume that the only nontrivial commutation relations are
\bqn \label{eq:B}
[X_{i},X_{j}]=\sum_{h=1}^{m}b_{ij}^{h}Y_{h},
\eqn
where $B_{1},\ldots,B_{m}$ defined by $B_{h}=(b_{ij}^{h})$ are skew-symmetric matrices (see \cite{corank2} for the role of these matrices in the exponential map).

Due to the group structure, the intrinsic sub-Laplacian takes the form of sum of squares $\lapl=\sum_{i=1}^{k}X_{i}^{2}$ (see Remark \ref{r:group}). The group structure also implies that the heat kernel is invariant with respect to the group operation
hence it is enough to consider the heat kernel $p_{t}(0,q)$ starting from the identity of the group, which we also denote $p_{t}(q)$. The heat kernel is written as follows (see again \cite{bealsgaveaugreiner,lanconellibook})
\bqn\label{eq:gaveau}\nn
\ \ p_{t}(q)=\frac{2}{(4\pi t)^{Q/2}} \int_{\R^{m}} V(B(\tau)) \exp \left(-\frac{ W(B(\tau)) x\cdot x}{4t}\right) \cos \left(\frac{z \cdot \tau}{t}\right) d\tau,
\eqn
where  $q=(x,z)$, $x\in \R^k, z\in \R^m$, and $B(\tau):=\sum_{i=1}^{m} \tau_{i}B_{i}$. Moreover  $V: \R^{n\times n} \to \C$ and $W: \R^{n\times n} \to  \R^{n\times n}$ are the matrix functions defined by
$$V(A)=\sqrt{\det \left(\frac{A}{\sin A}\right)}, \qquad W(A)= \frac{A}{\tan A}.$$
Here $Q$ is the Hausdorff dimension of the sub-Riemannian structure.

Notice that \eqref{eq:gaveau} differs by some constant factors from the formulas contained in  \cite{bealsgaveaugreiner} since there the heat kernel is the solution of the equation $\partial_t=\frac12 \lapl$. 
\brem
Assume that the real skew-symmetric matrix $B(\tau)$ is diagonalizable and denote by $\pm i \lam_{j}(\tau)$, for $j= 1,\ldots \ell$, its non zero eigenvalues. Then we have the formula for the expansion on the \virg{vertical subspace} (i.e.\ where $x=0$)
\bqn \label{eq:z}
p_{t}((0,z))=\frac{2}{(4\pi t)^{Q/2}} \int_{\R^{m}} \prod_{j=1}^{\ell} \frac{\lam_{j}(\tau)}{\sinh \lam_{j}(\tau)}\cos \left(\frac{z \cdot \tau}{t}\right) d\tau.
\eqn
\erem
\subsection{The Heisenberg group}\label{s:h}
The Heisenberg group is the simplest example of sub-Riemannian
manifold. It is defined by the orthonormal frame
$\distr=\tx{span}\{X_{1},X_{2}\}$ on  $\R^{3}$ (with coordinates
$(x,y,z)$) defined by
$$X_{1}=\partial_{x}-\frac{y}{2}\partial_z, \qquad
X_{2}=\partial_{y}+\frac{x}{2} \partial_z.$$
Defining $Z=\partial_{z}$, we have the commutation relations
$[X_{1},X_{2}]=Z$ and $[X_{1},Z]=[X_2,Z]=0.$
Denote by
$\EXP_{0} :\Lambda_{0}\times \R^{+} \to M$ the exponential map starting from the origin, where
$$\Lambda_{0}=\{p_{0}=(\theta,w)\in T^{*}_0M\, |\, \theta\in S^1, \ w\in \R\}.$$
For every $p_0=(\theta,w)\in \Lambda_{0}$ with $|w|\neq0$, the
arclength geodesic $\g(t)=\EXP_0(p_0,t)=(x(t),y(t),z(t))$ associated with the
initial covector $p_{0}$
is described by the equations
\begin{align} \label{eq:exp1}
x(t)&=\frac{1}{ w}(\cos ( w t +\theta) - \cos
\theta),\notag \\
y(t)&=\frac{1}{w}(\sin ( w t +\theta) - \sin\theta),\\
z(t)&=\frac{1}{2w^2}(wt- \sin w t ). \notag
\end{align}
and  is optimal up to its cut time $\tcut=2\pi/w$, with $\g(\tcut)=(0,0,\pi/w^2)$.
If $w=0$, the geodesic is a straight line
contained in the $xy$-plane and $\tcut=+\infty$.

From these properties it follows that the cut locus starting from the origin coincides with the $z$-axis, and for every point $\zeta=(0,0,z)$ in this set we have
$d^2(0,\zeta)=4\pi |z|$.
 
\brem
The expression of the heat kernel $p_t(q)$ for the Heisenberg group is
well known and was first computed by Gaveau \cite{gaveau} and Hulanicki \cite{hulanicki}.
The integral formula for $p_t$ can be directly recovered from
\eqref{eq:gaveau} since in this case there is a single skew-symmetric
matrix $B$
$$B=\begin{pmatrix}
0&1\\ -1&0
\end{pmatrix}, \qquad \tx{eig}(B(\tau))=\{\pm \, i \tau\}.$$
Hence it follows
\bqn
\hspace{0.5cm} p_{t}(0,q)=\frac{2}{(4\pi t)^2}\int_{-\infty}^{\infty}\frac{\tau}{\sinh\tau}
\exp\left(-
\frac{x^2+y^2}{4 t} \frac{\tau}{\tanh\tau}
\right)\cos\left(\frac{z\tau}{t}\right) d\tau. \label{eq:heispt} \nn
\eqn
On the vertical axis the integral can be explicitly computed
\begin{align*}
p_{t}(0,\zeta)&=\frac{2}{(4\pi t)^2}\int_{-\infty}^{\infty}\frac{\tau}{\sinh\tau} \cos\left(\frac{z\tau}{t}\right) d\tau=\frac{1}{8 t^2} \frac{1}{1+\cosh \left(\frac{\pi  z}{t}\right)}.
\end{align*}
Hence, using that $d^2(0,\zeta)=4\pi z$ we have
\begin{align} \label{eq:lequazione}
p_{t}(0,\zeta)&=\frac{1}{t^2} \exp\left(-\frac{\pi z}{t}\right) \psi(t) 
=\frac{1}{t^2} \exp\left(-\frac{d^2(0,\zeta)}{4t}\right) \psi(t),
\end{align}
where $\psi(t)$ is a smooth function of $t$, nonvanishing at 0. (Here $z$ is fixed.)
 \erem
 
In what follows, we recover the expansion \eqref{eq:lequazione} computing the expansion of the hinged energy function and applying Corollary \ref{t:main}.
For reasons of symmetry it is not restrictive to consider only  points $\hat{\zeta}=(0,0,\hat{z})$ such that $\hat{z}>0$ (the on-diagonal expansion
is a different situation).

The set of
minimal geodesics joining $0$ to $\hat{\zeta}=(0,0,\hat{z})$ is parametrized by the covectors $p_0=(\theta,\hat{w})$ where $\theta\in S^1$, $\hat z=\pi/\hat w^2$. For each $p_0$, the associated geodesic $\g_{p_0}$ satisfies $\g_{p_0}(0)=0$ and $\g_{p_0}(2\pi/\hat w)=\hat \zeta$.
Further, we have that the set of midpoints $\Gamma$ is characterized as follows
$$\Gamma=\EXP_0\lp S^1, \hat w, \frac{\pi}{\hat w}\rp=\lc \lp x,y,\frac{\hat z}{2}\rp :
x^2+y^2=2/\hat{w} \rc.$$ 

We introduce cylindrical coordinates $(\rho,\phi,z)$, where $x=\rho \cos \phi, \ y=\rho \sin \phi$. 
We have that $(t,\theta,w)$ forms a smooth coordinate system on $\GE$ (for
$\GE$ small), where $t$ represents the distance from the origin. Because of the invariance with respect to rotation around the $z$ axis, to compute the Hessian of the hinged energy function $h_{0,\hat{\zeta}}$
we are left to study the relationship
between $(\rho,z)$ and $(t,w)$ near $\Gamma$. We have
\[
\rho(t,w)=
\frac{2}{w}\sin\lp\frac{wt}{2}\rp, \quad\text{and}\quad
z(t,w) = \frac{1}{2w^2}\lp wt-\sin wt\rp .
\]
Recall that
\[
 h_{0,\hat{\zeta}}(\rho,z) = \frac{1}{2}\lp d^2(0,(\rho,z)) + d^2((\rho,z),\hat{\zeta}) 
\rp.
\]

Using that $t$ respresents the distance from the origin and exchanging the role of $0$ and $\hat \zeta$, one can get with some implicit differentiation for the matrix element of the Hessian of $h_{0,\hat \zeta}$ 
\bqn \label{eq:hessh} \nn
\left.\frac{\partial^2}{\partial z^2}
h_{0,\hat \zeta}(\rho,z)\right|_{\Gamma} =
2\hat{w}^2,\quad
\left.\frac{\partial^2}{\partial \rho^2}
h_{0,\hat \zeta}(\rho,z)\right|_{\Gamma} = \frac{\pi^2}{2} , \quad
\left.\frac{\partial^2}{\partial \rho \partial z}
h_{0,\hat \zeta}(\rho,z)\right|_{\Gamma} = 0.
\eqn

It follows that there exists a smooth change of coordinates
$(\rho,z)\mapsto (u,v)$ on a small disk perpendicular to $\Gamma$ (with
respect to the the usual $\bR^3$ metric) with the following three
properties.  First, $\Gamma$ corresponds to the set where $u$ and $v$
are both zero.  Second, $ h_{0,\hat \zeta}(u,v) =
\frac{\pi^2}{\hat{w}^2} + u^2+v^2$ on $\GE$.  Third,
$du=\frac{\pi}{2}d\rho$ on $\Gamma$ and $dv=\hat{w}dz$ on $\Gamma$. Applying Theorem
\ref{t:mainmain}  and keeping track of all the constants one gets
\[
p_t(0,\hat \zeta) =
\frac{1}{t^2}\exp\left( -\frac{\pi^2/\hat{w}^2}{t}\right)
\lp \frac{48\pi \lp c_0(0,\Gamma)\rp^2}{\hat{w}^2} +O(t) \rp .
\]
where $c_0(0,\Gamma)$ is the constant apearing in the Ben Arous expansion. Taking into account that $\hat z=\pi/\hat w^2$, the heat kernel decays like a constant times
$t^{-2}\exp(-4\pi \hat z / 4t)$, which agrees with what one obtains from
equation \eqref{eq:lequazione}.

\subsection{ (3,6) case} \label{s:36}
The free nilpotent Lie group $(3,6)$ is the sub-Rieman\-nian structure on $\R^{6}$ (with coordinates $(x_{1},x_{2},x_{3},z_{1},z_{2},z_{3})$) defined by the distribution $\distr=\tx{span}\{X_{1},X_{2},X_{3}\}$, where the vector fields 
$$X_{1}=\partial_{x_1}-\frac{1}{2}x_2\partial_{z_{3}}+\frac{1}{2}x_3\partial_{z_{2}},$$
$$X_{2}=\partial_{x_2}+\frac{1}{2}x_1\partial_{z_{3}}-\frac{1}{2}x_3\partial_{z_{1}},$$
$$X_{3}=\partial_{x_3}+\frac{1}{2}x_2\partial_{z_{1}}-\frac{1}{2}x_1\partial_{z_{2}},$$
define an orthonormal frame. If we set $Z_{i}=\partial_{z_{i}}$ for $i=1,2,3$ we have
$[X_{1},X_{2}]=Z_{3}$, $[X_{2},X_{3}]=Z_{1}$, and $[X_{3},X_{1}]=Z_{2}.$

In this case the matrices $B_{k}=(b_{ij}^{k})$ defined by the identities $[X_{i},X_{j}]=b_{ij}^{k}Y_{k}$ are 
$$B_{1}=\begin{pmatrix}
0&0&0\\0&0&1\\0&-1&0\\
\end{pmatrix},
\qquad
B_{2}=\begin{pmatrix}
0&0&1\\0&0&0\\-1&0&0\\
\end{pmatrix},
\qquad
B_{3}=\begin{pmatrix}
0&1&0\\-1&0&0\\0&0&0\\
\end{pmatrix}.$$
and for their linear combination
$B(\tau)=\sum_{j=1}^{3} \tau_{j}B_{j}$ we have $\tx{eig}(B(\tau))=\left \{0,\pm\,i|\tau|\right\},$ where
we denote by $|\cdot|$ the standard norm on $\R^{3}$. 

Using \eqref{eq:z} the explicit expression on the \virg{vertical} subspace, i.e. at a point $\zeta=(0,0,0,z_{1},z_{2},z_{3})$ 
is written as follows
\bqn \label{eq:heat36vert}
p_{t}(\zeta)=\frac{2}{(4\pi t)^{9/2}}\int_{\R^{3}} \frac{|\tau|}{\sinh |\tau|} \cos\left (\frac{\tau \cdot z}{t}\right)d\tau, 
\eqn
To compute the expansion of the heat kernel for $t\to 0$ we use the fact that \eqref{eq:heat36vert} is the Fourier transform of the radial function $f(\tau)=|\tau|/\sinh |\tau|$.

Recall that, if $F(x)=f(|x|)$ is a radial function defined on $\R^{m}$, its Fourier transform $\wh{F}(\xi)$ is itself a radial function,  i.e.\ it is defined by $\wh{F}(\xi)=g(|\xi|)$, where $g$ is the function of one variable that satisfies
$$g(\rho)=\frac{(2\pi)^{m/2}}{\rho^{\frac{m-2}{2}}}\int_{0}^{\infty} J_{\frac{m-2}{2}}(\tau \rho) \tau^{m/2}f(\tau) d\tau, \qquad \rho=|\xi|,$$
and $J$ denotes the Bessel function. In our case $m=3$, we have $J_{1/2}(s)=\sqrt{\frac{2}{\pi s}}\sin s$ and
$$g(\rho)=4\pi\int_{0}^{\infty} \frac{\sin \rho \tau}{\rho \tau}f(\tau) \tau^{2}d\tau,$$
Then we can rewrite our heat kernel as the 1-dimensional integral
\bqn
 \label{eq:heat36vert2} \nn
p_{t}(\zeta)=\frac{8\pi }{(4\pi t)^{9/2}}\int_{0}^{\infty} \frac{\tau^{2 }\sin \rho \tau}{\rho\, \sinh \tau}d\tau, 
\qquad \text{where}\qquad\rho=\frac{|z|}{t}.
\eqn
Using that
\bqn\nn
\int_{0}^{\infty} \frac{\tau^{2 }\sin \rho \tau}{\rho\, \sinh \tau}d\tau=\frac{2 \pi ^3 \sinh ^4\left(\frac{\pi  \rho}{2}\right) }{\rho\, \text{sinh}^3(\pi 
   \rho)}, \qquad \rho\in \R,
\eqn
we can explicitly write the expression of the heat kernel for $\zeta$ such that $|z|=1$
\bqn
 \label{eq:heat36vert3}
p_{t}(\zeta)=\frac{\sinh ^4\left(\frac{\pi }{2 t}\right) }{32 \sqrt{\pi } t^{7/2} \text{sinh}^3\left(\frac{\pi
   }{t}\right)}.
\eqn
From \eqref{eq:heat36vert3} one can immediately show that for such $\zeta$
\bqn\label{eq:36}
\lim_{t\to 0} t^{7/2}e^{\frac{\pi}{t}}p_t(\zeta)=C>0. 
\eqn
The following lemma is a direct consequence of Theorem \ref{THM:Leandre}:
\bl\label{l:36}
Assume that there exist $\alpha,K>0$, $t_{0}>0$  and constants $C_{1},C_{2}>0$ such that
\bqn \label{eq:c1c2}
\frac{C_{1}}{t^\alpha}e^{-\frac{K}{4t}} \leq p_t(x,y) \leq \frac{C_{2}}{t^\alpha}e^{-\frac{K}{4t}}, \qquad \all 0\leq t \leq t_{0}.
\eqn
Then $K=d^2(x,y)$.
\el
\begin{proof} Since log is a monotone function, we can apply $4t \log$  both inequalities in \eqref{eq:c1c2} and letting $t\to 0^+$ (which is allowed since the estimate is uniform for small $t$) we have
$\lim_{t\to 0^+} 4t \log p_t(x,y)=-K$
and the statement follows from Theorem \ref{THM:Leandre}. 
\end{proof} 
 
\bp Let $\zeta=(0,0,0,z_1,z_2,z_3)$ with $|z|=1$. Then $d^2(0,\zeta)=4\pi$ and the following asymptotic expansion holds
\bqn \nn
p_t(\zeta)=\frac{1}{t^{7/2}}\exp \lp-\frac{d^2(0,\zeta)}{4t}\rp \phi(t),
\eqn
where $\phi(t)$ is a smooth function nonvanishing at $t=0$. Moreover $\zeta$ is a conjugate point. 
\ep
\begin{proof}
This follows directly from \eqref{eq:36}, Lemma \ref{l:36} and Corollary 
\ref{corollarioMAIN}. 
\end{proof}
\brem 
From this analysis of the heat kernel and the homogeneity of the distance one gets the following information: 
(i). $d^2(0,\zeta)=4\pi|z|$ for every $\zeta=(0,0,0,z_1,z_2,z_3)$. (ii). The point $\zeta$ is reached from the origin by an optimal geodesic that at time $t=\sqrt{4 \pi |z|}$ is also conjugate. 

These facts were proved in \cite{nilpotent36} with a detailed analysis of the exponential map. 
(Notice that by symmetry is not difficult to prove that the point is conjugate to the origin along the geodesic. 
On the contrary, the difficulty is in proving that the geodesic does not lose optimality before the conjugate locus.) 
Our method via the analysis of the heat kernel provides a shorter proof.
\erem

\section{Grushin plane}
\label{s:gru}
The Grushin plane is the generalized sub-Riemannian structure  on $\R^2$ for which an orthonormal frame of vector fields is given by
\bqn \label{eq:gru}
X=\partial_x,\qquad Y=x\partial_y.
\eqn
Since $Y$ vanishes on the $y$-axis, this is a rank-varying sub-Riemannian structure and in particular is a 2-dimensional almost-Riemannian structure (see Appendix). One immediately verifies that the Lie bracket generating condition is satisfied since $[X,Y]=∂_y$.

In this section we compute the expansion of the heat kernel in the Grushin plane at a conjugate point, starting from a Riemannian point.

The interesting feature of this structure is that it provides an example of almost Riemannian geometry in which the geodesic flow is completely integrable by means of trigonometric functions and, at the same time, the conjugate locus has the same structure of the conjugate locus of a generic 2-dimensional Riemannian metric.

The sub-Riemannian Hamiltonian associated with the orthonormal frame \eqref{eq:gru} (in standard coordinates $\lam=(p_x,p_y,x,y)$ in $T^*\R^2$) is the smooth function
\bqn\label{eq:H}
H:T^*\R^2 \to \R, \qquad H(p_x,p_y,x,y)=\frac{1}{2}(p_x^2+x^2p_y^2).
\eqn
Since in this case there are no abormal minimizers (see \cite{agrachevboscainsigalotti})
the arclength geodesic flow starting from the Riemannian point $q_0=(-1,-\pi/4)$ is computed as the solution of the Hamiltonian system associated with $H$, with initial condition $(x_0,y_0)=(-1,-\pi/4)$ and $(p_x(0),p_y(0))=(\cos \theta, \sin \theta)$, where  $\theta\in S^1$. The exponential map $\EXP: \R_+\times S^1 \to \R^2$ starting from $q_0$, is computed as follows (we omit the base point $
q_0$ in the notation)
\begin{align}
\EXP(t,\theta)&=(x(t,\theta),y(t,\theta)),\nn \\[0.2cm]
x(t,\theta)&=-\frac{\sin (\theta-t \sin \theta)}{\sin \theta},  \label{eq:bauendi1}\\
y(t,\theta)&=-\frac{\pi}{4}+\frac{1}{4\sin \theta} \left(2t-2\cos \theta +\frac{\sin (2 \theta-2 t \sin \theta)}{\sin \theta}\right),\nn
\end{align}
with the understanding $\EXP(t,0)=\lim_{\theta \to 0}\EXP(t,\theta)=(t-1,-\pi/4)$.

\begin{figure}[h!]
  \centering
      \includegraphics[width=0.8\textwidth]{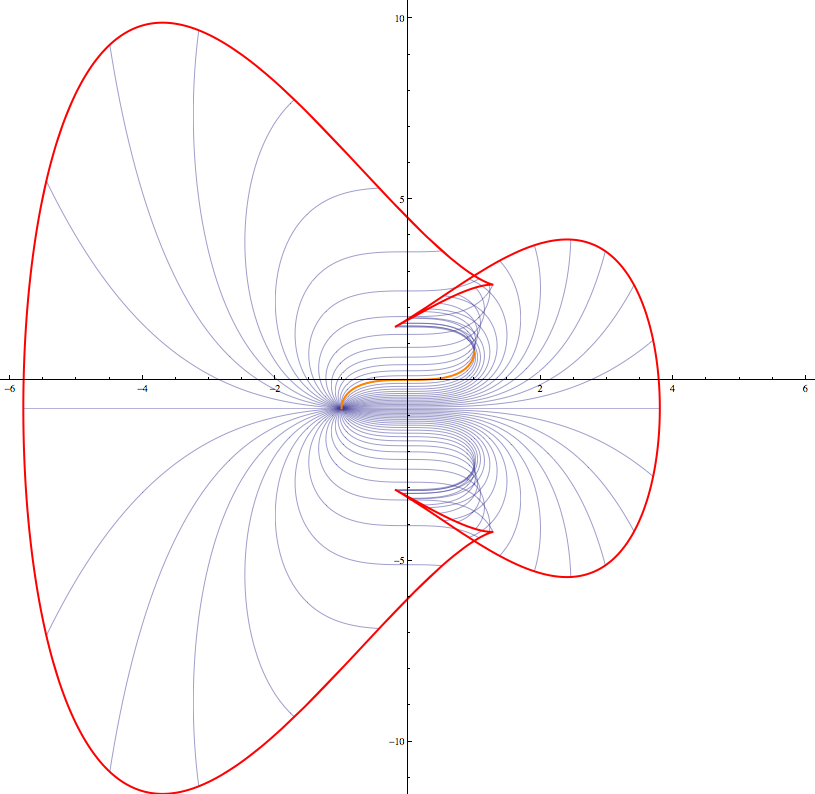}
  \caption{Geodesics starting from $q_0$}
  \label{fig:gru}
\end{figure}
 Let us consider the point $q_1=(1,\pi/4)$, the symmetric of $q_0$ with respect to the origin. 
The point $q_1$ is both a cut and a conjugate point from $q_0$. Indeed from the results of \cite{agrachevboscainsigalotti} immediately follows that the cut locus from $q_0$ is the set $\text{Cut}(q_0)=\{(1,\pi/4+s),s\geq0\}$. Moreover  
\bqn \label{eq:evvai}
\EXP(\pi,\pi/2)=q_1, \qquad \frac{d}{d\theta}\bigg|_{\theta=\pi/2}\EXP(\pi,\theta)=(0,0),
\eqn
shows that $q_1$ is also conjugate to $q_0$. Figure \ref{fig:gru} shows some geodesics starting from the point $q_{0}$ and the endpoints of all geodesics starting from $q_{0}$ at time $T=4.8$.
\brem 
Notice that
the geodesic with initial covector $\theta=\pi/2$ is the only one that reach $q_1$ optimally in time $T=\pi/2$. The midpoint of the geodesic is the origin $\EXP(\pi/4,\pi/2)=(0,0)$. (See also Figure \ref{fig:conj}.)
\erem
\begin{figure}[h!]  
 \centering
      \includegraphics[width=0.35\textwidth]{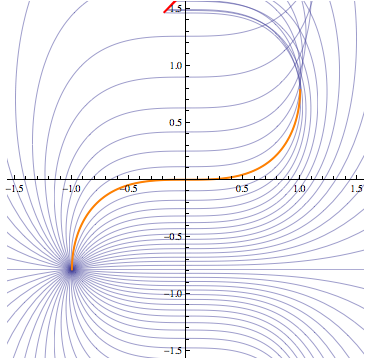}
  \caption{The conjugate geodesic}
  \label{fig:conj}
\end{figure}

We are interested in the small time asymptotic expansion of $p_t(q_0,q_1)$, where $p_t$ denotes the heat kernel of the sub-Riemannian heat equation
$$\partial_t \phi=\lapl \phi, \qquad \lapl=X^2+Y^2=\partial_x^2+x^2\partial_y^2.$$

Here the sub-Laplacian is not the intrinsic one but is computed with respect to the standard Lebegue measure of $\R^2$. Indeed in this case the intrinsic volume $\mu=\frac{1}{|x|} dxdy$ is diverging along the singular set $\mc{Z}=\{x=0\}$ hence our results does not apply since $\mu$ is not smooth. (See \cite{boscain-laurent} for a discussion of the intrinsic heat equation in the Grushin plane.)

An integral representation for the heat kernel for the operator $\partial_x^2+x^2\partial_y^2$ can be easily obtained by computing the Fourier transform with respect to the $y$ variable and then using the Mehler kernel for the quantum harmonic oscillator.  Its expression, given $q=(x,y),q'=(x',y')$ is
\begin{align*}
p_t(q,q')=\frac{1}{(2\pi t)^{3/2}}&\int_{-\infty}^{\infty}    
\exp\left(\frac{x x'}{t}\frac{ \tau  }{\text{sinh}(2 \tau )}-\frac{\left(x^2+x'^2\right)
  }{2 t}\frac{ \tau }{ \tanh (2 \tau )}\right) \times\\
&\times \sqrt{\frac{\tau}{\text{sinh}(2 \tau )}}
   \exp\left( \frac{i \left(y-y'\right) \tau }{t} \right)
 d\tau\nn.
\end{align*}
However from this formula it seems hard to find an asymptotic expansion for $t$ small except on the diagonal at the origin.

Thanks to Corollary \ref{t:main}, to compute the asymptotic expansion of the heat kernel $p_t(q_0,q_1)$ we are reduced to study the expansion of the hinged energy function $h_{q_0,q_1}$ near the origin (we omit the points in the notation in what follows) 
\begin{align*}
h(x,y)&=\frac{1}{2}\left(d^2(q_0,(x,y))+d^2(q_1,(x,y))\right) \\
&=\frac{1}{2}\left(d^2(q_0,(x,y))+d^2(q_0,(-x,-y))\right),
\end{align*}
where the last identity follows from the symmetries of the structure and implies that the expansion of $h$ at the origin contains only even order terms in $(x,y)$ and we are reduced to compute the even terms of the expansion of the function $(x,y)\mapsto d^2(q_0,(x,y))$.

\brem
By \eqref{eq:evvai} and the proof of Theorem \ref{THM:HAndConj} it follows that the Hessian of the hinged energy function $h$ is degenerate along the direction $(-1,1)$ since
$$\frac{d}{d\theta}\bigg|_{\theta=\pi/2}\EXP(\pi/2,\theta)=(-1,1).
$$
For this reason we consider the new coordinate system $(\bar x,\bar y)$ around 0 defined by
$\bar x=\frac{x+y}{2}, \quad \bar y=\frac{x-y}{2}$.
The Hessian of $h$ is diagonal in these coordinates.
\erem

Using the fact that the geodesics defined by \eqref{eq:bauendi1} are  parametrized by arclength, we can compute the derivatives of the distance with respect to $(\bar x,\bar y)$ by computing derivatives of $t$ from \eqref{eq:bauendi1} with implicit differentiation (as in Section \ref{s:h}). After some computations one finds the following expansion for $h$ (we omit the bar in $\bar x,\bar y$ for the new system of coordinates)
\begin{align} \label{eq:hhh}
h(x,y)=4x^2+\frac{\al-32}{24}&x^4-\frac{\al}{6}x^3y+\frac{\al-32}{4}x^2y^2\\
&-\frac{\al}{6}xy^3+\frac{\al}{24}y^4+O(\|(x,y)\|^5),\nn
\end{align}
where $\al=\frac{3}{2}\pi^2$.

Concerning our hinged energy function \eqref{eq:hhh} one can also show that the following explicit change of coordinates 
$$\phi(u,v)=(u+\frac{32 - \al}{192} u^3+\frac{\al}{48} v^3+\frac{\al}{48} u^2 v+\frac{16-\al}{32} u v^2,v).$$
 diagonalizes $h$ up to order 5. Namely
$$h(\phi(u,v))=8u^2+\frac{\al}{24}v^4+O(\|(u,v)\|^5).$$

A direct application of Corollary \ref{t:main} (recall also Lemma \ref{Lem:Splitting}), together with $d(q_0,q_1)=\pi/2$, gives
\bt The heat kernel $p_t(q_0,q_1)$ satisfies the following asymptotic expansion
\bqn \label{eq:q1}
p_t(q_0,q_1)=\frac{1}{t^{5/4}}e^{-\frac{\pi^2}{16t}}(C+O(t)).
\eqn
\et
\brem
Notice that the same expansion as in \eqref{eq:q1} holds for the symmetric point $q_2=(1,-3\pi/4)$.  If $q\notin\{q_1,q_2\}$  
$$p_t(q_0,q)\sim\frac{1}{t}e^{-\frac{d^2(q_0,q)}{4t}}(C+O(t)).$$
\erem

\brem
Corollary  \ref{corollarioMAIN} can be applied to compute the heat kernel asymptotics starting from the origin. In this case the cut locus is the $y$ axes and these points are not conjugate. 
On the diagonal, applying the Leandre - Ben Arous result \eqref{eq:formula1} with $Q=3$ (or using the explicit formula for the heat kernel given above),
one gets $p_t((0,0),(0,0))\sim C/t^{3/2}$ with $C>0$. Off diagonal, applying Corollary \ref{corollarioMAIN}, one gets  $p_t((0,0),(x,y))\sim C(x,y)/t$ for some $C(x,y)>0$.
\erem
The expansion of the heat kernel for the Grushin plane is summarized in the following table:

\medskip
\hspace{-3cm}
\begin{tabular}{|c|c|c|c|c|}\hline

                                        &diagonal&off diagonal& off diagonal   &off diagonal\\
                                        
                                        & (Leandre)& off cut locus&cut (non-conjugate)                                 &cut conjugate\\
                                        &(Ben Arous)  &(Ben Arous) &(Corollary {\bf 1})  &(Corollary  {\bf 2})\\
                                        \hline
&&&&\\
$p_t(q,q')$                       &$\sim \frac{C}{t}$&$\sim \frac{C}{t}e^{-d^{2}(q,q')/(4 t)}$&$\sim \frac{C}{t}e^{-d^{2}(q,q')/(4 t)}$&$\sim \frac{C}{t^{5/4}}e^{-d^{2}(q,q')/(4 t)}$\\
$q$ Riemannian point    &&&&\\\hline
&&&&\\
$p_t(q,q')$                       &$\sim \frac{C}{t^{3/2}}$&$\sim \frac{C}{t}e^{-d^{2}(q,q')/(4 t)}$&$\sim \frac{C}{t}e^{-d^{2}(q,q')/(4 t)}$&---\\
$q$ degenerate point      &&&&\\\hline
\end{tabular}

\appendix
\section{Extension to rank-varying sub-Riemannian structures}
\label{appendix}

In this section we give a more general definition of sub-Riemannian manifold (that we call rank-varying sub-Riemannian manifold). This definition includes also as a particular case Riemannian manifolds. For a more complete presentation one can see \cite{nostrolibro}. All the results of the paper hold for this more general structure.

Let $M$ be an $n$-dimensional smooth manifold.  Given a vector bundle $\E$ over $M$, the $\con^\infty(M)$-module of smooth sections of $\E$ is denoted by $\Gamma(\E)$. For the particular case  $\E=TM$, the set of smooth vector fields on $M$ is denoted by $\VecM$. 

\begin{definition}\label{fiberrvd}
An {\it $(n,k)$-\rvd} on an $n$-dimensio\-nal manifold $M$ is 
a pair $(\E,\f)$ where $\E$ is a vector bundle of rank $k$ over $M$ and $\f:\E\rightarrow TM$ is a morphism of vector bundles, i.e.\ {\bf (i)}  the diagram 
$$
\xymatrix{
 \E  \ar[r]^{\f} \ar[dr]_{\pi_\E}   & TM \ar[d]^{\pi}            \\
 & M                          
}    
$$
 commutes, where  $\pi:TM\rightarrow M$ and $\pi_\E:\E\rightarrow M$ denote the canonical projections and {\bf (ii)} $\f$ is linear on fibers.
Moreover, we require the map $\sigma\mapsto\f\circ\sigma$ from $\Gamma(\E)$ to $\VecM$ to be injective.
\end{definition}

Let $(\E,\f)$ be an $(n,k)$-\rvd, $\bD=\{\f\circ\sigma\mid\sigma\in\Gamma(\E)\}$ be its associated submodule and denote by $\bD_q$  the linear subspace $\{V(q)\mid  V\in \bD\}=\f(\E_q)\subseteq  T_q M$.
Let  $\mathrm{Lie}(\bD)$ be the smallest Lie subalgebra of
$\mathrm{Vec}(M)$
containing $\bD$ and, for every $q\in M$, let $\mathrm{Lie}_q(\bD)$ be the linear subspace of $T_qM$ whose elements are the evaluation at $q$ of elements belonging to $\mathrm{Lie}(\bD)$.
We say that $(\E,\f)$ satisfies the {\it  H\"ormander condition} if
$\mathrm{Lie}_q(\bD)=T_q M$ for every $q\in M$.

\begin{definition}\label{gensrs}
An {\it $(n,k)$-\rvsR}  is a triple $(\E,\f,\langle\cdot,\cdot\rangle)$ where $(\E,\f)$ is a Lie bracket generating $(n,k)$-\rvd\ on a  manifold $M$ and $\langle\cdot,\cdot\rangle_q$ is a   scalar product on $\E_q$ smoothly depending on $q$. 
\end{definition}

Several classical structures can be seen as particular cases of \rvsR s, e.g.,  Riemannian structures (when $\E=TM$ and $f=id$) and constant-rank sub-Riemannian structures (as defined in Section \ref{s:srg}). An $(n,n)$-\rvsR\ is called an {\it $n$-dimensional almost-Riemannian structure}. An example of 2-almost Riemannian structure is provided by the Grushin plane, see \cite{agrachevboscainsigalotti,agrboschaghe}. 

If $\sigma_1,\dots,\sigma_k$ is an orthonormal frame for $\langle\cdot,\cdot\rangle$ on an open subset $\Omega$ of $M$, an {\it  orthonormal frame} in $\Omega$ for the \rvsR\ is given by $X_1,\ldots,X_k$, where $X_i:=\f\circ\sigma_i$. Orthonormal frames are systems of local generators of $\bD$. 
For every $q\in M$ 
and every $v\in\bD_q$ define
\bqn
\Gq(v)=\inf\{\langle u, u\rangle_q \mid u\in \E_q,\f(u)=v\}.
\eqnn
Notice that if ${X_1,\ldots,X_k}$ is an orthonormal frame for the \rvsR\ in $\Omega$, then it may happen that there exist a $q\in \Omega$ such that $\dim \text{span}\{X_1(q),\ldots,X_k(q)\}<k $ and that $\Gq(X_i(q))<1$ for some $i$.

A Lipschitz continuous curve $\g:[0,T]\to M$ is said to be \emph{horizontal} (or \emph{admissible}) 
if there exists a measurable essentially bounded function 
\bqn
[0,T]\ni t\mapsto u(t)\in \E_{\g(t)},
\eqnn
called  {\it control function}, 
such that 
$\dot \g(t)=\f(u(t))$  for almost every $t\in[0,T]$. 
Given an admissible 
curve $\g:[0,T]\to M$, the {\it length of $\g$} is  
\bqn
\ell(\g)= \int_{0}^{T} \sqrt{ {\bf G}_{\gamma(t)}(\dot \g(t))}~dt.\eqnn

The {\it Carnot--Caratheodory distance} is defined as
\bqn\nonumber
d(q_0,q_1)=\inf \{\ell(\g)\mid \g(0)=q_0,\g(T)=q_1, \g\ \mathrm{admissible}\}.
\eqn

As in the classical sub-Riemannian case, the \hp\ of connectedness of $M$ and the H\"ormander condition guarantees the finiteness and the continuity of $d(\cdot,\cdot)$ with respect to the topology of $M$.
 
For rank-varying sub-Riemannian structures the  definitions of minimizers, geodesics, normal and abnormal extremals and the formulation of the Pontryagin Maximum Principle are the same as in the constant rank case. Also the definition of cut and conjugate loci are the same. 
Thanks to the injectivity assumption, the definition of the horizontal gradient is still $\grad \phi= \sum_{i=1}^k X_i(\phi) X_i$.  
  The definition of the Popp's volume is instead more delicate, since the volume diverges while approaching a point in which there is a drop of rank of the distribution.
However, for a smooth volume $\mu$ the sub-Laplacian still has the form $\lapl=\sum_{i=1}^k X_i^2 + (\dive X_i)X_i$, and all the results of the paper hold in this case.

\vspace{0.2cm}
{\bf Acknowledgements.} The authors would like to thank
Fabrice Baudoin and Andrei Agrachev for helpful discussions. The authors also thank IHP for its hospitality during the finishing of this paper.
{\small
\bibliographystyle{siam}
\bibliography{SRNeel-Biblio}
}
\end{document}